\documentclass[letterpaper]{imanum}

\newcommand{\lemref}[1]{Lemma \ref{lem:#1}}
\newcommand{\remref}[1]{Remark \ref{rem:#1}}
\newcommand{\thmref}[1]{Theorem \ref{thm:#1}}
\newcommand{\corref}[1]{Corollary \ref{cor:#1}}
\newcommand{\defref}[1]{Definition \ref{def:#1}}
\newcommand{\secref}[1]{Section \ref{sec:#1}}

\newcommand{\figref}[1]{Figure \ref{fig:#1}}
\newcommand{\tabref}[1]{Table \ref{tab:#1}}
\newcommand{\equref}[1]{(\ref{eq:#1})}

\newcommand{\condref}[1]{\ref{cond:#1}}
\newcommand{\figloc}[1]{\emph{(#1)}}

\usepackage[all]{nowidow}
\usepackage{graphicx}
\usepackage{amsmath,amsfonts,mathtools}
\usepackage{xcolor}
\usepackage{nowidow}
\usepackage{enumitem}
\usepackage[hidelinks]{hyperref}
\AtBeginDocument{
  \label{CorrectFirstPageLabel}
  \def\fpage{\pageref{CorrectFirstPageLabel}}
}
\usepackage{tikz-cd}

\usepackage{tabularx}
\usepackage{array}
\newcolumntype{L}[1]{>{\raggedright\arraybackslash} m{#1} }
\newcolumntype{C}[1]{>{\centering\arraybackslash} m{#1} }

\DeclareMathAlphabet{\mathcal}{OMS}{cmsy}{m}{n}
\newcommand{\Surf}{\Gamma} \newcommand{\dSurf}{{\Gamma_h}} \newcommand{\Lap}{\Delta_\Surf}  \newcommand{\grad}{\nabla_\Surf} \newcommand{\dgrad}{\nabla_\dSurf}  
   \newcommand{\ltpr}[2]{{\left( #1, #2 \right)_0}}
\newcommand{\ltpra}[2]{{\left( #1, #2 \right)_{0,h}}}
\newcommand{\hopr}[2]{{\left( #1, #2 \right)_1}}
\newcommand{\hopra}[2]{{\left( #1, #2 \right)_{1,h}}}
\newcommand{\hilbertpr}[2]{{\left\langle #1, #2 \right\rangle}}
\newcommand{\tup}[2]{{\left( #1, #2 \right)}}
\newcommand{\ciarlsmooth}{V}
\newcommand{\ciarldiscronsurf}{\widetilde{V}^h}
\newcommand{\ciarldiscronmesh}{V^h}
\newcommand{\ciarlfunc}{J}
\newcommand{\ciarlfunca}{J_h}
\newcommand{\rieszsmooth}{\mathcal{R}}
\newcommand{\rieszdiscronsurf}{\widetilde{\rieszsmooth}^h}
\newcommand{\rieszdiscronmesh}{\rieszsmooth^h}
\newcommand{\discrlap}{\widetilde{L}^h}
\newcommand{\discrlapa}{L^h}
\newcommand{\Fsmooth}{F}
\newcommand{\Fdiscronsurf}{\widetilde{F}^h}
\newcommand{\Fdiscronmesh}{F^h}
\newcommand{\discrspacesmap}{W^h}
\newcommand{\usmooth}{u}
\newcommand{\udiscronsurf}{\tilde{u}^h}
\newcommand{\udiscronmesh}{u^h}
\newcommand{\secondderiv}{D^2_{\Surf}}

\newcommand{\secondderiva}{D^2_{\Surf\!\!,\,h}}
\DeclareMathOperator*{\argmin}{argmin}
\newcommand{\R}{\mathbb{R}}
\newcommand{\Rtwo}{\mathbb{R}^2}
\newcommand{\Rthree}{\mathbb{R}^3}

\newcommand{\Ho}{{H^1}}
\newcommand{\Hoa}{{H^1_h}}
\newcommand{\Ht}{{H^2}}

\newcommand{\tHt}{{\tilde{H}^2}}

\newcommand{\Hf}{{H^4}}
\newcommand{\Hk}{{H^k}}
\newcommand{\Hoo}{{H^1_0}}\newcommand{\Hto}{{H^2_0}}\newcommand{\tHoo}{{\tilde{H}^1_0}}
\newcommand{\Hooa}{{H^1_{0,h}}}\newcommand{\hatSh}{\hat{S}_h}
\newcommand{\Sh}{{S_h}}
\newcommand{\Shlt}{{S_{h; \Lt}}}
\newcommand{\hatSho}{\hat{S}_{h,0}}
\newcommand{\hatf}{\hat{f}}
\newcommand{\Sho}{{S_{h,0}}}\newcommand{\Lp}{{L^p}}

\newcommand{\Lt}{{L^2}}
\newcommand{\Lta}{{L^2_h}}
\newcommand{\Linf}{{L^\infty}}
\newcommand{\Wkp}{{W^{k,p}}}\newcommand{\Woinf}{{W^{1,\infty}}}\newcommand{\tWoinf}{{\tilde{W}^{1,\infty}}}

\newcommand{\Woinfa}{{W^{1,\infty}_h}}\newcommand{\Wtinf}{{W^{2,\infty}}}\newcommand{\tWtinf}{{\tilde{W}^{2,\infty}}}

\newcommand{\Wtp}{{W^{2,p}}}

\newcommand{\norm}[1]{{\left\lVert #1 \right\rVert}}

\newcommand{\abs}[1]{{\left| #1 \right|}}

\newcommand{\id}{\operatorname{Id}}
\newcommand{\Ih}{{I_h}\,}
\newcommand{\Rh}{{R_h}\,}
\newcommand{\Rho}{{R_{h,0}\,}}\newcommand{\Rha}{{R_h^{(h)}}}
\newcommand{\Rhao}{{R_{h,0}^{(h)}}\,}\newcommand{\hatuoh}{{\hat{u}_1^h}}
\newcommand{\uoh}{{u_1^h}}
\newcommand{\hatuth}{{\hat{u}_2^h}}
\newcommand{\uth}{{u_2^h}}
\newcommand{\uoht}{{\tilde{u}_1^h}}
\newcommand{\utht}{{\tilde{u}_2^h}}
\newcommand{\dd}{\operatorname{d}}
\newcommand{\dx}{\dd x}

\newcommand{\defeq}{\vcentcolon=}

\jno{drnxxx}
\received{19 December 2020}

\begin{document}

\title{A mixed finite element method with piecewise linear elements for the
biharmonic equation on surfaces}
\shorttitle{Linear mixed FEM for the biharmonic equation on surfaces}

\author{{\sc Oded Stein\thanks{Corresponding author. Email: ostein@mit.edu}} \\[2pt]
Massachusetts Institute of Technology, MA, USA, and\\
Columbia University, NY, USA\\[6pt]
{\sc and}\\[6pt]
{\sc Eitan Grinspun} \\[2pt]
University of Toronto, ON, Canada, and\\
Columbia University, NY, USA\\[6pt]
{\sc and}\\[6pt]
{\sc Alec Jacobson} \\[2pt]
University of Toronto, ON, Canada\\[6pt]
{\sc and}\\[6pt]
{\sc Max Wardetzky} \\[2pt]
University of G\"ottingen, Germany\\\quad
}

\shortauthorlist{O. Stein \emph{et al.}}

\maketitle

\begin{abstract}
{

The biharmonic equation with Dirichlet and Neumann boundary conditions
discretized using the mixed finite element method and
piecewise linear (with the possible exception of boundary triangles)
finite elements
on triangular elements has been well-studied for domains
in \(\Rtwo\).
Here we study the analogous problem on polyhedral surfaces.
In particular, we provide a convergence proof of discrete solutions to the
corresponding smooth solution of the biharmonic equation. 
We obtain convergence rates that are identical to the ones known for the planar
setting.
Our proof focuses on three different problems:
solving the biharmonic equation on the surface,
solving the biharmonic equation in a discrete space in the metric of the
surface, and
solving the biharmonic equation in a discrete space in the metric of the 
polyhedral approximation of the surface.
We employ inverse discrete Laplacians to bound the error between the solutions
of the two discrete problems, and generalize a flat strategy to bound the
remaining error between the discrete solutions and the exact solution on the
curved surface.
 }
{biharmonic equation; polyhedral surfaces; mixed finite elements; discrete geometry.}
\end{abstract}

\section{Introduction}
\label{sec:introduction}
We consider the \emph{biharmonic equation} on smooth
surfaces embedded in three-dimensional Euclidean space:
given a function \(f\) on a smooth surface \(\Surf\)
with smooth boundary \(\partial\Gamma\),
find a function \(u\) such that
\begin{equation}
    \Lap^2 u = f
    \;\textrm{,}
\end{equation}
where \(\Lap\) is the
Laplace--Beltrami operator on the smooth surface \(\Surf\).
This Laplacian arises from the Riemannian metric \(g\) on \(\Surf\), where \(g\)
is inherited from ambient three-dimensional Euclidean space.
If boundaries are present, boundary conditions must be taken into account (where we assume that $\Surf$ has a smooth boundary).
We consider Dirichlet and Neumann boundary conditions,
\begin{equation}
    u = \frac{\partial u}{\partial \mathbf{n}} = 0
    \quad\quad
    \textrm{at the boundary,}
\end{equation}
where \(\frac{\partial u}{\partial \mathbf{n}}\) denotes the
co-normal
derivative of \(u\) at the boundary -- the scalar product of the function's
gradient and the boundary normal.
For flat surfaces, this problem is sometimes referred to as the
\emph{clamped thin plate problem}.
If no boundaries are present,  \(f\) and \(u\) must have zero mean, i.e., 
		\(\int_\Surf f \; \dx  = \int_\Surf u \; \dx = 0\).\\

In this paper, we use a \emph{mixed approach}, which
corresponds to solving a linear system of two equations:
in \(u_1\) (which corresponds to the solution \(u\)), and in
\(u_2\) (which corresponds to the Laplacian \(\Lap u\) of the solution).
Given \(f\in \Lt(\Surf)\), the (smooth) mixed formulation takes the following
form: Find \(u_1 \in \Hoo(\Surf), \; u_2 \in \Ho(\Surf)\) such that
\begin{equation}\label{eq:intro:smoothmixedformulation}
	\hopr{u_2}{\xi} = \ltpr{f}{\xi} \quad\forall \xi \in \Hoo(\Surf) \quad\text{and}\quad
	\hopr{u_1}{\eta} = \ltpr{u_2}{\eta} \quad \forall \eta \in \Ho(\Surf)
	\;\textrm{.}
\end{equation}
Here the Sobolev spaces \(\Lt(\Surf)\) and \(\Hoo(\Surf)\) are equipped with the
inner products 
\begin{equation}\label{eq:intro:smoothInnerProducts}
	\ltpr{u}{v} =  \int_\Surf u v \;\dx \quad \text{and} \quad \hopr{u}{v} =
    \int_\Surf g\left(\grad u, \grad v\right) \;\dx
	\;\textrm{,}
\end{equation}
respectively, where \(g\) denotes the Riemannian metric on \(\Surf\),
and \(\grad\) denotes the gradient on \(\Surf\).
The mixed method can be formulated for any \(u_1 \in \Hoo(\Surf)\) and \(u_2 \in \Ho(\Surf)\),
and has a unique solution such that
\(u_1 \in \Hoo(\Surf) \cap \Hf(\Surf)\) and \(u_2 \in \Ht(\Surf)\),
as we assumed smooth \(\partial\Gamma\)
\cite[]{GazzolaPolyharmonic}.\\

In order to solve \equref{intro:smoothmixedformulation} numerically, we use a
corresponding mixed finite element method on a polyhedral surface $\dSurf$ that
is nearby the smooth surface $\Surf$ in the sense of conditions \textbf{(C1-C4)}
from \secref{discretesurface}.
In particular, we consider $\dSurf$ to be a mesh with piecewise flat triangles
with straight edges
(with the exception of boundary triangles, which can have curved edges along the boundary -- see \figref{shortestdistance}), together with a bijection $\Psi: \dSurf \to \Surf$ that is defined via the closest point projection.
As finite elements we use the space $\hatSh$ of
\emph{piecewise linear elements} on $\dSurf$
(with the exception of boundary triangles, which can have modified
elements).
Then the \emph{discrete} mixed formulation on $\dSurf$ is:
Find \(\hatuoh \in \hatSho, \; \hatuth \in \hatSh\) such that
\begin{equation}\label{eq:intro:discreteproblemexplicitlymesh}
\int_\dSurf \dgrad \hatuth \cdot \dgrad \xi \dx =
    \int_\dSurf \hatf \xi \dx \;\;\;\; \forall \xi \in \hatSho
    \quad\text{and}\quad
    \int_\dSurf \dgrad \hatuoh \cdot \dgrad \eta \dx =
    \int_\dSurf \hatuth \eta \dx \;\;\;\; \forall \eta \in \hatSh
    \;\textrm{,}
\end{equation}
where $\hatf = f\circ \Psi$ is the evaluation of the right hand side $f$ from \equref{intro:smoothmixedformulation} on $\dSurf$, and \(\hatSho\defeq \hatSh\cap \Hoo(\dSurf)\).
\\

In our approach, every mesh $\dSurf$ is required to have \emph{uniformly shape-regular triangles}. Moreover, we consider sequences of triangle meshes that converge to a given 
smooth limit surface such that both \emph{positions} and \emph{normals} converge at a certain rate (to be specified later).
Letting $\uoh = \hatuoh \circ \Psi^{-1}$ and $\uth = \hatuth \circ \Psi^{-1}$ denote the \emph{liftings} of the discrete solutions from $\dSurf$ to $\Surf$, we show:
\begin{itemize}
    \item \(\Lt\)-convergence of $\uoh$ to $u_{1}$ of order \(h\)
    (\thmref{u1error});
    \item \(\Ho\)-convergence of $\uoh$ to $u_{1}$ of order \(h^\frac{3}{4}\)
    (\corref{u1h1error});
    \item \(\Lt\)-convergence of $\uth$ to $u_{2}$ of order \(\sqrt{h}\)
    (\thmref{u2error});
\end{itemize}
where \(h\) is the maximum edge length of the approximating triangles of the mesh.
If no boundaries are present, the problem becomes simpler and we observe that
better convergence rates can be obtained.\\

Before continuing with an overview of our proof, we provide a few historical remarks on mixed finite elements and finite elements for curved surfaces for context.
\paragraph*{Mixed finite elements for the biharmonic equation in \(\Rtwo\).}
\citet{Ciarlet1974} introduce the mixed finite element
method for the biharmonic problem. Their work informs the functional analysis framework that we use here.
They solve the same system of equations that we end up solving
(in the flat case), but only show convergence for higher-order (\(\geq 2\))
Lagrangian finite elements.
Their approach is later expanded by \citet{Monk1987} to deal with boundary smoothness
problems caused by triangulating (in the flat case).

\citet{Scholz1978} proves that the mixed finite element method
with linear, first-order Lagrange elements can be used to solve the biharmonic problem, and
he gives an error bound of \(h \log^2 h\) in the \(L^2\) norm of the solution.
The result by Scholz is central to understanding the convergence of the
linear finite element method for the biharmonic equation, and
forms the basis of our proof.
The result is remarkable, since it shows convergence of the method, even
though the standard convergence conditions for mixed finite elements
(the \emph{inf-sup} conditions \cite[]{BoffiMixedFEM})
are not fulfilled.
Scholz's error estimate is not optimal, as it relies on an \(L^\infty\) estimate
of the Ritz projection
error by \citet{Nitsche1978}.
An application of a later, better bound for the same interpolation error
\cite[]{Rannacher1982} gives convergence of order \(h\).

\citet{Oukit1996} provide an analysis of the biharmonic equation with Dirichlet
and Neumann boundary conditions that combines a hybrid approach
(which they call Hermann--Miyoshi) and the mixed approach (which they call
Ciarlet--Raviart).
Their analysis holds for first and second order elements.
The result by \citet{Scholz1978} is recovered in the limit
\(p \rightarrow \infty\), where \(p\) is the degree of the \(L^p\) space
used in their estimate (5.39).
An alternative approach to solving the biharmonic equation using mixed finite
elements is the decomposition into four linear equations, such
as done by \citet{Behrens2011} (which leads to superconvergence of
the solution) and \citet{Li2017}.

\paragraph*{Finite elements on curved surfaces.}
\citet{Dziuk1988} generalizes the standard result for solving the Poisson
equation with linear finite elements from \(\mathbb{R}^2\) to smooth surfaces by working with \emph{inscribed} meshes, i.e.,
requiring that vertices of the approximating mesh be contained in the limit
surface.
His approach to analyzing discretizations of curved surfaces has since been used in advanced numerical methods for curved surfaces
\cite[]{Dziuk2007,Demlow2007,Olshanskii2009,Du2011,Dziuk2013b}.
An overview of methods to discretize the Laplace--Beltrami operator on curved
surfaces can be found, for example, in the works of
\citet{Dziuk2013a,Bonito2019}.

\citet{Wardetzky2006} and \citet{Hildebrandt2006} generalize Dziuk's result to non-inscribed meshes.
We also work with this generalized setting here, as non-inscribed meshes are
prevalent in various applications, e.g., in geometry processing.
More specifically, the setting that we consider here, namely discretizations of the 
biharmonic equation
(and its related Helmholtz problem) using a mixed formulation with linear
Lagrange elements have been popular in practice
\cite[]{Desbrun1999,Sorkine2004,Bergou2006,Garg2007,Tosun2008,Jacobson2010,Jacobson2011,Jacobson2012}.
In this article, we provide a justification for the use of the linear mixed
finite element method for such applications. 

Meanwhile, other methods for solving the biharmonic equation on surfaces exist
in the literature.
\citet{larsson2017} use the discontinuous Galerkin approach to achieve a method for
surfaces without boundary where the \(L^2\) error is of order \(h\).
\citet{cockburn2016} use a different kind of discontinuous Galerkin method as
well as non-conforming mixed finite elements (Raviart--Thomas,
Brezzi--Douglas--Marini, Brezzi--Douglas--Fortin--Marini).
\citet{Elliott2019} analyze a generalized framework for the solution of a
variety of fourth-order problems formulated as saddle point problems in a
Dziuk-like setting.
Fourth-order problems and their discrete formulations also arise in fluid dynamics, see, e.g., \cite[]{Reisken2018}.

\paragraph*{Proof strategy.}
The above mentioned  \emph{liftings} $\uoh$ and $\uth$ of discrete Lagrange
functions $\hatuoh$ and $\hatuth$ from $\dSurf$ to $\Surf$ via the closest
point projection enable us to compare discrete and smooth solutions.
This is necessitated by the geometric difference between the triangle mesh and
the smooth limit surface (since otherwise there would be no way to compare the
two solutions).
Our proof requires certain approximation properties in order for the closest point
projection to yield a bijection.
These requirements are fulfilled, e.g., for \emph{inscribed} meshes as
considered by \citet{Dziuk1988}, i.e., the
case where mesh vertices reside on the smooth surface and the triangles are
uniformly shape regular.
Our setting is more general than the case of inscribed meshes.
We detail our setting in \secref{discretization}.

We denote by \(\Sh\) and \(\Sho\) the finite
element spaces $\hatSh$ and $\hatSho$ lifted to the smooth surface \(\Surf\)
(using the closest point projection $\Psi$), and we denote by
\(\ltpra{\cdot}{\cdot}\) and \(\hopra{\cdot}{\cdot}\) the \(\Lt\) and \(\Hoo\)
inner products lifted from \(\dSurf\) to \(\Surf\).
Using these lifted inner products, \equref{intro:discreteproblemexplicitlymesh} becomes
\begin{equation}\label{eq:intro:discreteproblemmesh}
    \hopra{\uth}{\xi} = \ltpra{f}{\xi} \quad\forall \xi \in \Sho \quad\text{and}\quad
    \hopra{\uoh}{\eta} = \ltpra{\uth}{\eta} \quad \forall
    \eta \in \Sh
    \;\textrm{.}
\end{equation}
However, lifting alone does not suffice for directly proving error estimates for
\((u_1 - \uoh)\) and  \((u_2 - \uth)\).
Indeed, a problem arises from the difference between the metric on the
polyhedral surface (which is piecewise flat) lifted to \(\Surf\)
and the smooth metric \(g\).
This difference implies that the Hilbert spaces \(\left(\Sho, \hopra{\cdot}{\cdot}\right)\) and
\(\left(\Sh, \hopra{\cdot}{\cdot} + \ltpra{\cdot}{\cdot}\right)\) for the discrete solutions are \emph{not} subspaces of the
Hilbert spaces \(\left(\Hoo(\Surf), \hopr{\cdot}{\cdot}\right)\) and
\(\left(\Ho(\Surf), \hopr{\cdot}{\cdot} + \ltpr{\cdot}{\cdot}\right)\) for the smooth solutions since, although the sets are respectively subsets, the metrics differ.

Therefore, we introduce the
following \emph{auxiliary} discrete mixed problem:
find \(\uoht \in \Sho, \; \utht \in \Sh\) such that
\begin{equation}\label{eq:intro:discreteproblemsurface}
    \hopr{\utht}{\xi} = \ltpr{f}{\xi} \quad\forall \xi \in \Sho \quad\text{and}\quad
    \hopr{\uoht}{\eta} = \ltpr{\utht}{\eta} \quad \forall
    \eta \in \Sh
    \;\textrm{,}
\end{equation}
where the inner products are those arising from the smooth metric \(g\). Using this approach, the spaces \(\left(\Sh, \ltpr{\cdot}{\cdot}\right)\) and \(\left(\Sho, \hopr{\cdot}{\cdot}\right)\) are indeed subspaces of the Hilbert spaces \(\left(\Lt(\Surf), \ltpr{\cdot}{\cdot}\right)\) and \(\left(\Hoo(\Surf), \hopr{\cdot}{\cdot}\right)\), respectively. Notice that this problem is only an auxiliary problem for our proof; its operators are never computed in practice.
\\

Considering the auxiliary mixed problem is central to our approach, since it allows us to adapt the proof of \citet[]{Scholz1978}, which treats the case of convergence for the mixed formulation of the biharmonic problem using linear elements for the case of flat domains in \(\Rtwo\). In the planar case, one has  \(\uoh=\uoht\) and \(\uth=\utht\) by construction. \citet{Scholz1978} splits up the proof for the planar case into showing that \(\uth\) converges to \(u_2\) and  that \(\uoh\) converges to \(u_1\).

In the curved case, a similar argument only works to show that \(\utht\) converges to \(u_2\), and that \(\uoht\) converges to \(u_1\). Indeed, in order to bound the error between  \(u_2\) and \(\uth=\utht\)  in the flat setting, Scholz invokes an \(L^\infty\) estimate for the
Ritz projection.
His result relies on a suboptimal bound by \citet{Nitsche1978}.
An application of a later, better bound for the same interpolation error
\cite[]{Rannacher1982} yields the above-mentioned convergence rates in the flat setting. In order to adapt this analysis to the curved setting, we rely on an \(L^\infty\) estimate for the
Ritz projection
provided by \citet{Demlow2009}.
While \citet{Demlow2009} works with inscribed meshes, his result can be adapted to our more general setup, resulting in a bound of the error between \(\utht\) and \(u_2\) when combined with Scholz's approach. 

In our setup, we must consider an additional step in order to bound the error between  \(\uth\) and \(\utht\), which finally leads to a bound on the error between \(\uoh\) and \(u_1\). To this end, we adapt the classical formulation for the mixed biharmonic problem introduced by \citet{Ciarlet1974}, which requires the definition of the following function spaces:
\begin{equation}\begin{split}\label{eq:intro:mixedspacesdefn}
		\ciarlsmooth &:= \{ \tup{v_1}{v_2} \in \Hoo \times \Lt \quad |
        \quad \hopr{v_1}{\mu} = \ltpr{v_2}{\mu} \; \forall \mu \in \Ho \}
        \;\textrm{,} \\
        \ciarldiscronsurf &:= \{ \tup{v_1}{v_2} \in \Sho \times \Shlt \quad |
        \quad \hopr{v_1}{\mu} = \ltpr{v_2}{\mu} \; \forall \mu \in \Sh \}
        \;\textrm{,} \\
        \ciarldiscronmesh &:= \{ \tup{v_1}{v_2} \in \Sho \times \Shlt \quad |
        \quad \hopra{v_1}{\mu} = \ltpra{v_2}{\mu} \; \forall \mu \in \Sh \}
        \;\textrm{,}
    \end{split}\end{equation}
where the space \(\Shlt\) is the space \(\Sh\), but with the \(\Lt\) norm instead of its usual \(\Ho\) norm. Using these spaces, we formulate the mixed biharmonic problem employing the Riesz map that results from an inner product that is different from the spaces' product metric. We detail this construction in \secref{mixedfem}. Notice that the space \(\ciarldiscronsurf\) is absent in the classical formulation of Ciarlet and Raviart, as it corresponds to our auxiliary mixed problem. 

It requires beyond a mere generalization of Scholz's proof to control the error
between \(\uth\) and \(\utht\).
We do so by bounding the
\emph{geometric error} between the function spaces \(\ciarldiscronmesh\) and
\(\ciarldiscronsurf\). While the need for controlling geometric errors is also
present when studying solutions to the Poisson equation on surfaces
\cite[]{Dziuk1988,Wardetzky2006}, these results cannot be directly applied to
our setting.
We bound the difference between elements of the spaces \(\ciarldiscronmesh\) and
\(\ciarldiscronsurf\) by using inverse discrete Laplacians for constructing a map between these function spaces. 
Using inverse discrete
Laplacians introduces an error of order \(h^{-1}\), which we manage to contain
using a geometric error bound  of order \(h^{\frac 32}\), adapted from the work of
\citet{Wardetzky2006}. This yields an error of order \(h^{\frac 12}\), which exactly
corresponds to the error estimate for \(u_2\) given in \citet{Scholz1978} for
the planar case.

\paragraph*{Overview.}
We detail the mixed biharmonic problem of the smooth setting in
\secref{definitions}.
We then continue with the description of discrete function spaces and
differential operators on polyhedral surfaces in \secref{discretization},
where we consider the case of non-inscribed meshes and their relation to smooth
surfaces that they discretize.
In \secref{mixedfem}, we continue with the description of the approach of
\citet{Ciarlet1974} adapted to our setting.
In \secref{convergence} we provide our convergence proof for the mixed
biharmonic problem on curved surfaces.

\section{The Biharmonic Equation on Smooth Surfaces}
\label{sec:definitions}

Let
\(\Surf\)
be a compact smooth surface with smooth boundary or no boundary,
embedded into \(\Rthree\).
We denote by \(\Lp\) the usual \(\Lp\)-spaces on surfaces, and we let 
\(\Wkp\) denote the Sobolev space with \(k\) weak derivatives in \(\Lp\).
We let \(\Hk \defeq W^{k,2}\), and we denote by
\(\Hoo \subseteq \Ho\) the subspace of functions with zero trace along the
boundary (for surfaces with nonempty boundary), or  those functions that have
zero mean (for surfaces without boundary).
Whenever the domain is omitted, these spaces are implied to be defined over
a smooth surface
\(\Surf\).

We denote the metric tensor on
\(\Surf\)
by \(g(\cdot,\cdot)\), i.e., the restriction of the inner product on 
\(\Rthree\) to the tangent spaces of
\(\Surf\).
The metric induces the \(\Lt\) inner product
\(\ltpr{\cdot}{\cdot}\)
and the \(\Hoo\) inner product
\(\hopr{\cdot}{\cdot}\),
\begin{equation}\begin{split}\label{eq:nonDistortedInnerProducts}
	\ltpr{u}{v} &= \int_\Surf u v \;\dx 
	\quad\quad\quad\quad\quad\; u,v \in \Lt \;\textrm{,} \\
	\hopr{u}{v} &= \int_\Surf g\left(\grad u, \grad v\right) \;\dx
	\quad\quad u,v \in \Ho
	\;\textrm{.}
\end{split}\end{equation}
The norm on \(\Ho\) is induced by the inner product
\(\ltpr{\cdot}{\cdot} + \hopr{\cdot}{\cdot}\).

\begin{definition}
	\label{def:biharmeqndefn}
	For \(f \in \Lt\), the \emph{biharmonic equation} is defined
	as follows:
	find \(u \in \Hoo \cap \Hf\), such that
	\begin{equation}
		\Lap^2 u = f
		\;\textrm{,}
	\end{equation}
	where
	\(\Lap\)
	is the positive semidefinite Laplace--Beltrami operator on
	\(\Surf\).
	Additionally,
	\begin{itemize}
		\item if
		\(\Surf\)
		has a boundary, then zero Dirichlet and Neumann
		boundary conditions apply,
		\(u = \frac{\partial u}{\partial\mathbf{n}} = 0\);
		\item if
		\(\Surf\)
		is closed, \(f\) must have zero mean, i.e.,
		\(\int_\Surf f \; \dx  = 0\).
		\\
	\end{itemize}
\end{definition}

The biharmonic equation has a corresponding weak formulation.
	For \(f \in \Lt\), find \(u \in \Hto\) such that
	\begin{equation}\label{eq:weakbiharmonic}
		\int_\Surf \Lap u \Lap v \;\dx
		= \int_\Surf f v \;\dx
		\quad\quad \forall v \in \Hto
		\;\textrm{,}
	\end{equation}
	where \(\Hto\) is the subspace of \(\Ht\) with zero Dirichlet and Neumann  boundary conditions.
	If \(\Surf\) is a closed surface, one additionally requires that
	\(\int_\Surf f \;\dx = 0\), and one looks for \(u\) such that
	\(\int_\Surf u \;\dx = 0\).

We assume that there is a unique solution such that
\(u \in \Hf, \; \norm{u}_\Hf \leq C \norm{f}_\Lt\).
For closed surfaces this follows from the fact that the biharmonic equation
decouples into
two Poisson equations (given that \(f\) integrates to zero), and 
for planar domains, it follows from \citet[Section 2.5.2]{GazzolaPolyharmonic}.
Additionally, we assume the standard existence and regularity estimates for the
Poisson equation: for \(g \in \Lp\), \(1<p<\infty\) there is a unique
\(w \in \Wtp\) with Dirichlet boundary conditions such that, weakly, \(\Lap w = g\)
and \( \norm{w}_\Wtp \leq C \norm{g}_\Lp\).
See, for example,
the work of \citet[Section 2]{grisvard} for planar domains
or \citet{Dziuk2013a} for smooth surfaces.

With \(u_1 \defeq u \), and using the intermediate variable
\(u_2 \defeq \Lap u_1\),
\equref{weakbiharmonic} can be rewritten in its mixed form
\cite[(1.4)]{Monk1987}
\begin{equation}\begin{split}\label{eq:smoothmixedformulation}
	\hopr{u_2}{\xi} &= \ltpr{f}{\xi} \quad\quad \forall \xi \in \Hoo
	\;\textrm{,} \\
	\hopr{u_1}{\eta} &= \ltpr{u_2}{\eta} \quad\quad \forall \eta \in \Ho
	\;\textrm{.}
\end{split}\end{equation}
We refer to this system of equations as the smooth mixed formulation of the biharmonic equation
with Dirichlet and Neumann boundary conditions
(this system was also mentioned in \equref{intro:discreteproblemsurface}).
It can be formulated for any \(u_1 \in \Hoo\) and \(u_2 \in \Ho\).
By \citet[Theorem 7.1.1]{CiarletTheFiniteElementMethod}, \equref{smoothmixedformulation} has a unique
solution such that \(u_1 \in \Hoo \cap \Ht\), so by our assumptions
that $\Surf$ is a smooth surface with smooth boundary
the mixed problem has
a unique solution such that \(u_1 \in \Hoo \cap \Hf\) and \(u_2 \in \Ht\).  
\\

\section{Discretization}
\label{sec:discretization}

\subsection{Discretizing the surface}
\label{sec:discretesurface}

In the discrete setting, we work with a triangulated surface, i.e., a connected topological manifold of dimension two, piecewise consisting of flat triangles.
Boundary edges of triangles along the surface boundary are allowed to be curved
as long as the curve remains in the plane of the triangle.
In the
planar case, where \(\Surf \subseteq \Rtwo\) is a flat surface
embedded in the
plane only (and not, as in our general case, embedded in \(\Rthree\)),
triangle meshes are only needed to discretize the function space
\(\Hoo\) in which the solution lives.
In the case of a surface
\(\Surf \subseteq \Rthree\), however, the mesh is
also used to discretize the geometry itself.
To deal with the error introduced by the discretization, we employ the
setting of \citet{Wardetzky2006}, which we explain in this section.

\begin{figure}
    \includegraphics{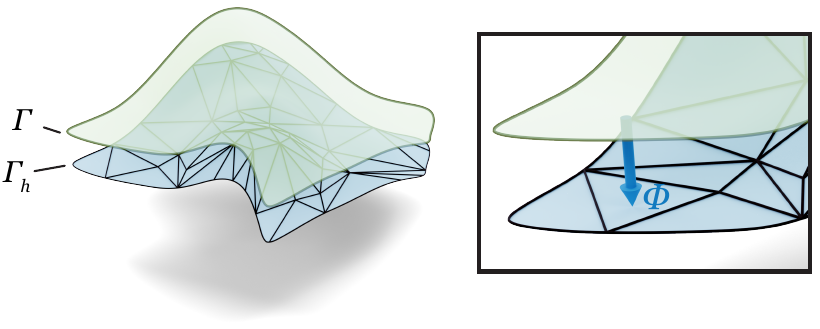}
    \caption{A surface \(\Surf\), in green, and an approximating mesh
    \(\dSurf\), in blue \figloc{left}.
    A close-up of the surface and the mesh with the inverse closest point projection
    \(\Phi = \Psi^{-1}\) between them.
    \label{fig:shortestdistance}}
\end{figure}

\begin{definition}[Reach]
	Let \(X\) be a topologically closed subset of \(\Rthree\).
	The medial axis of \(X\) is the set of those points in \(\Rthree\) that do
	not have a unique closest
	point
	in \(X\).
	The \emph{reach} of \(X\) is
	the distance of \(X\) to its medial axis,
	and we say that an object lies \emph{in the reach} of \(X\) if the object is
	closer to \(X\) than the medial axis of \(X\).
\end{definition}

Let \(\dSurf\)
be a triangle mesh, where all triangles are flat, and interior
triangles have straight edges while boundary triangles are allowed to have
curved edges along the boundary as long as these curved edges remain within the
triangle plane.
Then we can define the following map:

\begin{definition}
\label{def:closestpointprojection}
Let \(\dSurf\) lie within the reach of \(\Surf\).
The \emph{closest point projection} is the map
\(\Psi : \dSurf \rightarrow \Surf \) defined via 
\begin{equation*}
\Psi(q) = \argmin_{p \in \Surf} \norm{q-p}_{\Rthree} \ .
\end{equation*}
If \(\Psi\) is bijective, we define the inverse closest point projection as
\(\Phi = \Psi^{-1}\).
Notice that in this case, \(\Phi\) maps any \(p \in \Surf\) to the closest
intersection of the line through \(p\) parallel to the normal of \(\Surf\) at
\(p\) with \(\dSurf\), see \figref{shortestdistance}.

\end{definition}

Throughout, we require certain conditions of our mesh.
These conditions are \emph{automatically satisfied}
for any shape-regular triangle mesh that is  inscribed into a smooth surface
\(\Surf\) \cite[Section 3.5]{Wardetzky2006}.
Indeed, the following conditions are fulfilled by the setting considered in the
work of \citet{Dziuk1988} and others (with minor modifications at the
boundary).

\begin{enumerate}[label=\textbf{(C\arabic*)}]
	\item\label{cond:shaperegular} The triangles of
	\(\dSurf\)
	are uniformly
	shape regular, i.e., there exist constants \(\kappa, K > 0\) such that every
	triangle contains a circle of radius \(\kappa h\) and is contained in a
	circle of radius \(K h\).

	\item\label{cond:bijmapping} The polyhedral surface
	\(\dSurf\)
	is a \emph{normal graph} over the smooth surface
	\(\Surf\),
	i.e.,
	\(\dSurf\)
	lies within the reach of
	\(\Surf\)
	and the
	closest point projection \(\Psi\)
	is a bijective function.
	In particular, the boundary of
	\(\dSurf\)
	is bijectively mapped to the boundary of
	\(\Surf\), where triangle edges along the boundary of $\dSurf$ are allowed
	to be curved, but must remain in
	the plane of their respective triangle, see \figref{shortestdistance}.

	\item\label{cond:distancebound} The distance of every point under the
	closest point projection
	is bounded by \(C h^\gamma\) for some \(\gamma \geq \frac32\).\footnote{We adopt the convention that, wherever a constant \(C\)
	occurs, the words ``there is a constant \(C > 0\), dependent only on the
	surface \(\Surf\) and mesh regularity parameters'' are implied.}

	\item\label{cond:normalbound}
	Assuming \condref{bijmapping}, the triangle normals of
	\(\dSurf\) approximate the normals of \(\Surf\) in the sense that, at every
	point \(p \in \Surf\) that maps to an interior point  \(\Phi(p)\) of a triangle on \(\dSurf\), the angle between the surface normal of \(\Surf\) at \(p\) and the triangle normal of
	\(\dSurf\) at \(\Phi(p)\) is bounded by \(C h^\varepsilon\) for some
	\(\varepsilon \geq 1\).
\end{enumerate}
\begin{definition}\label{def:sigmadefn}
	We define the approximation parameter of the mesh \(\dSurf\) as
	\(\sigma \defeq \min\left(\gamma, 2\varepsilon\right) \geq \frac32\).
	Our bounds will depend on this combination of the mesh's pointwise
	and normal approximation quality.
\end{definition}

Condition \condref{bijmapping} might seem difficult to satisfy for meshes that have
nonempty boundary, since we require that the boundary of
\(\Surf\)
maps exactly to the boundary of
\(\dSurf\)
under
\(\Phi=\Psi^{-1}\).
However, this condition is similar to the condition of
\citet{Scholz1978} in the flat case:
Consider
a straight-edged triangle mesh within the reach of \(\Surf\). Let 
boundary vertices be inscribed into the boundary of \(\Surf\) such that
every triangle has at most two vertices on the boundary of the mesh.
For every boundary triangle \(T\), replace the straight boundary edge by a
curved edge in the plane of \(T\) such that the closest point projection
\(\Psi\) becomes surjective. This yields a piecewise flat surface \(\dSurf\)
for which condition \condref{bijmapping} is satisfied.
This makes condition \condref{bijmapping} very similar to the condition of
\citet{Scholz1978} in the flat case, which requires triangles with curved
edges that exactly match the boundary of the smooth domain.

\begin{remark}
	\label{rem:wardetzkybdry}
	Similar conditions to \textbf{(C1-C4)} are being used by
	\citet{Wardetzky2006} to show convergence of the finite element
	discretization of the Poisson equation.
	Condition \condref{bijmapping}, however, is not present in the work of
	\citet{Wardetzky2006} as it pertains to bijectivity at the boundary.
	Because of that, Wardetzky's result on the convergence of the finite element
	method for the Poisson equation, Theorem 3.3.3, only holds for solutions
	that are supported sufficiently far away from the boundary.
	With Condition \condref{bijmapping}, and with the finite element spaces that
	will be defined in \defref{femspaces}, the
	estimates of \citet{Wardetzky2006} for the finite element solutions
	of the Poisson equation
	extend from the case of solutions that are compactly supported away from the boundary to the general case of solutions \(u \in \Hoo \cap \Ht\).
	This mirrors similar work on polygonal meshes where
	boundary edges are allowed to be curved, as long as they bijectively
	map to the surface boundary \cite[]{Scholz1978,Demlow2007,Demlow2009}.
\end{remark}

\begin{remark}
	\label{rem:maybestraightboundaries}
	The numerical method described in this work uses triangles whose edges are
	allowed to be curved at the boundary to fulfill the bijectivity constraint
	from \condref{bijmapping}.
	It might be possible to relax condition \condref{bijmapping} to only require
	triangles with straight edges and vertices that are inscribed into the
	boundary.
	However, such a relaxation could lead to lower convergence rates.
	An example of such a construction in the flat case can be found,
	for example, in the work \cite{Monk1987}.
Other alternatives for weakening the bijectivity constraint
	from \condref{bijmapping} are given by the approach of \citet{Elliott2012}
	who employ a piecewise polynomial boundary, or by the approach of
	\citet{Burman2018} who weakly enforce a nonhomogeneous boundary
	condition on straight-edged boundary triangles.
\end{remark}

Using Conditions \textbf{(C1-C4)}, we can relate the metric and the function
spaces of the mesh
\(\dSurf\)
to the metric and the function spaces of
\(\Surf\).
Let \(g\) denote the metric tensor on the smooth surface \(\Surf\).
Notice that the polyhedral surface \(\dSurf\) can be regarded as a so-called
Riemannian cone manifold, see \cite{Troyanov1986}.
Indeed, the surface  \(\dSurf\) carries a Riemannian metric \(g_\dSurf\) that is
flat almost everywhere except at the mesh vertices, which are singularities for
the metric.
Notice in particular that the metric \(g_\dSurf\) is smooth across triangle
edges since any pair of adjacent triangles can be isometrically mapped to the
flat plane;
as a consequence, the metric \(g_\dSurf\)  does not ``see" triangle edges
(they are intrinsically flat). 

We can then use the inverse closest point projection \(\Phi\) in order to pull
back the cone manifold's metric \(g_\dSurf\) from \(\dSurf\)  to \(\Surf\).
This results in a metric \(g_h\) on  \(\Surf\), defined everywhere except at the preimage under \(\Phi\) of edges of \(\dSurf\).
More precisely, we define 
\begin{equation*}
		g_h(X,Y) \defeq g_{\dSurf}(\dd\Phi(X), \dd\Phi(Y))
		= g_{\Rthree} (\dd\Phi(X), \dd\Phi(Y)) \quad\text{a.e.,}
\end{equation*}
where \(X\) and \(Y\) are arbitrary smooth tangential
vector fields on \(\Surf\),
\(\dd\Phi\) is the Jacobian of \(\Phi\),
and where \(g_{\Rthree}\) denotes the standard
Euclidean metric of ambient three space.
Since \(\Gamma_h\) depends on the chosen triangle mesh, and \(\Phi\) depends on
the choice of \(\Gamma_h\), all three expressions depend on \(h\).

\begin{definition}\label{def:metricdefn}
Define the unique matrix field \(A\) on \(\Surf\) that relates
the pulled back metric \(g_h\) to the smooth metric \(g\) almost everywhere.
This matrix field is defined by requiring that
	\begin{equation}
		g_h(X,Y) = g(AX,Y) \quad \text{a.e.}
		\;
	\end{equation}
holds for all smooth vector fields \(X\) and \(Y\) on \(\Surf\).

Consider the  \(\Lt\) and \(\Hoo\) inner products on the polyhedral surface pulled back to the smooth surface \(\Surf\) via \(\Phi\). Using the matrix field \(A\), these inner products can be conveniently be expressed as 
\begin{equation}\label{eq:distortedInnerProducts}\begin{split}
		\ltpra{u}{v} &\defeq \int_\Surf uv \left| \det A \right|^\frac12 \;\dx
		\quad\quad\textrm{for }u,v \in \Lt(\Surf) \ , \\
		\hopra{u}{v} &\defeq \int_\Surf g(A^{-1} \grad u, \grad v)
		\left| \det A \right|^\frac12 \;\dx
		\quad\quad\textrm{for }u,v \in \Ho(\Surf)
		\;\textrm{,}
	\end{split}\end{equation}
respectively.
\end{definition}
We adopt the convention that for every norm, the same norm subscripted
with \(h\) implies that the norm is taken with respect to
the metric \(g_h\) lifted from \(\dSurf\) to \(\Surf\).
For example,
\(\norm{\cdot}_\Lta\)
is the \(\Lt\) norm in the metric \(g_h\).

\begin{remark}
	The discretization described in \secref{discretesurface} follows the
	approach of \citet{Wardetzky2006}, but parallels to some extent the theory
	of \citet{Dziuk1988} (who considers \emph{inscribed} meshes and
	\(\sigma=2\)).
	The metric distortion tensor \(A\) corresponds to a combination
	of \citeauthor{Dziuk1988}'s operators \(P\), \(P_h\), and \(I-\dd H\).
\end{remark}

The significance of using these inner products together with the lifting
defined by \(\Phi\) lies in the fact that this allows us to work on the smooth
surface \(\Surf\), even when considering operations on the polyhedral surface
\(\dSurf\), thus simplifying the comparison between solutions to differential
equations.
We cannot compute finite element operators on \(\Surf\) without numerical
integration, but we can compute them on the piecewise triangular \(\dSurf\)\footnote{with some exceptions at the boundary}.
Consequently, from now on we will exclusively work on the smooth
surface \(\Surf\).
In \secref{mixedfem}, this will allow us to introduce a discrete mixed
finite element problem (with solutions \(\uoh, \uth\)) on the surface
\(\Surf\) that is \emph{equivalent} to the discrete problem (with solutions
\(\hatuoh, \hatuth\)).	It is this new discrete mixed problem on \(\Surf\) with the modified
inner products from \equref{distortedInnerProducts} that allows us to
compute error bounds that also hold for \(\hatuoh, \hatuth\).\\

In order to bound certain geometric errors for our finite element spaces later on, we make use of explicit bounds on the entries of the matrix field \(A\) that describes the pulled back metric \(g_h\) in terms of the smooth metric \(g\).
From now on, the statement ``for small enough \(h\)'' is implied everywhere.
\begin{lemma}\label{lem:boundona}
	It holds that
	\begin{equation*}\begin{split}
		\norm{A - \id}_\Linf &\leq C h^\sigma \;\textrm{,} \\
		\norm{\abs{\det A}^\frac{1}{2} - 1}_\Linf &\leq C h^\sigma
		\;\textrm{,} \\
		\norm{\abs{\det A}^\frac{1}{2} A^{-1} - \id}_\Linf &\leq C h^\sigma
		\;\textrm{,}\\
		\norm{\abs{\det A}^{-\frac{1}{2}} A - \id}_\Linf &\leq C h^\sigma
		\;\textrm{,}
	\end{split}\end{equation*}
	where the \(\Linf\) norm is the essential supremum over the operator norms of the respective matrix fields.
	The scalar \(\sigma > 0\) depends on approximation properties of the mesh and is defined in
	\defref{sigmadefn}.
\end{lemma}
\begin{proof}
By \citet[Theorem 3.2.1]{Wardetzky2006}, for any point on
	\(\Surf\) where \(A\) is defined and for any orthonormal tangent frame
	there exists a matrix decomposition
	\(A = P Q P\) such that \(P,Q\) can be diagonalized
	(possibly in different bases) as
	\begin{equation*}\begin{split}
		P &= \begin{pmatrix} 1 - \phi\kappa_1 && 0 \\ 0 && 1 - \phi\kappa_2
		\end{pmatrix} \;\textrm{,} \\
		Q &= \begin{pmatrix} \frac{1}{\left(N \cdot N_h \right)^2} && 0 \\
		0 && 1 \end{pmatrix}
		\;\textrm{,}
	\end{split}\end{equation*}
	where \(\phi\) is the pointwise distance between
	\(\Surf\) and \(\dSurf\)
	under the map \(\Phi\), \(N\) and \(N_h\) are the surface
	normals of
	\(\Surf\) and \(\dSurf\)
	respectively,
	and \(\kappa_1, \kappa_2\) are the principal curvatures of
	the surface. Therefore,
	\begin{align*}
		\abs{\det A}^\frac{1}{2}  
		= \frac{|1-\phi\kappa_1| |1- \phi\kappa_2|}{\left|N \cdot N_h \right|} 	
		\approx \frac{|1-\phi\kappa_1- \phi\kappa_2|}{\left|N \cdot N_h \right|}
		\approx \frac{|1-\phi\kappa_1- \phi\kappa_2|}{|1-\angle\left(N, N_h \right)^{2}/2|} 
		\;\textrm{,}
	\end{align*}
	where \(\angle\left(N, N_h \right)\) denotes the unsigned angle between the two vectors
	\(N, N_h\) in \(\Rthree\) and we have dropped higher order terms in $\phi$ and  \(\angle\left(N, N_h \right)\).
	A simple Taylor expansion then gives
	\begin{equation*}
		\left|\abs{\det A}^\frac{1}{2} - 1\right| \approx
		\left|-\phi\kappa_1 - \phi\kappa_2 + \frac{1}{2} \angle\left(N, N_h \right)^2 \right|
		\;\textrm{,}
	\end{equation*}
	which proves the estimate for \(\norm{\abs{\det A}^\frac{1}{2} - 1}_\Linf\), given that
	\(\abs{\phi} \leq Ch^\gamma\) \condref{distancebound},
	\(\abs{\angle\left(N, N_h \right)} \leq Ch^\varepsilon\) \condref{normalbound},
	and \(\sigma = \min\left(\gamma, 2\varepsilon\right)\)
	(\defref{sigmadefn}).
	A similar argument works for the other three expressions.

\end{proof}

For inscribed meshes and \(\sigma=2\),
\lemref{boundona}, parallels the inequality on \(A_h\) found in
\citet[Section 5]{Dziuk1988}.

\begin{lemma}\label{lem:samefunctionspaces}
	Let \(\Lta(\Surf)\) denote the \(\Lt\) space on \(\Surf\) where integration
	happens with the volume element of the lifted metric \(g_h\).
	Let \(\Ho(\Surf), \Hoo(\Surf), \Woinf(\Surf)\) denote the
	\(\Ho, \Hoo, \Woinf\) spaces on \(\Surf\)
	where the gradient is taken with respect to the lifted metric \(g_h\),
	and integration happens with the volume element of the lifted metric
	\(g_h\)

	The following equalities hold as equalities of sets:
	\begin{equation*}\begin{split}
		\Lt(\Surf) &= \Lta(\Surf) \stackrel{\Phi}{=} \Lt(\dSurf) \ ,\\
		\Ho(\Surf) &= \Hoa(\Surf) \stackrel{\Phi}{=} \Ho(\dSurf) \ ,\\
		\Hoo(\Surf) &= \Hooa(\Surf) \stackrel{\Phi}{=} \Hoo(\dSurf) \ ,\\
		\Woinf(\Surf) &= \Woinfa(\Surf) \stackrel{\Phi}{=} \Woinf(\dSurf)
		\;\textrm{,}
	\end{split}\end{equation*}
	where $\stackrel{\Phi}{=}$ denotes the action of \(\Phi\) mapping between
	sets.

	The norms of the respective spaces are all equivalent independently
	of the choice of \(h\) (for \(h\) small enough).
\end{lemma}
\begin{proof}
	Since \(\Phi\) is bijective, every function on \(\Surf\)
	can be uniquely identified with a function on \(\dSurf\), using
	\(\Phi\) to lift functions.

	Under the map \(\Phi\), the inner products
	\(\ltpra{\cdot}{\cdot}\) and \(\hopra{\cdot}{\cdot}\)
	that generate the norms
	of the spaces
	\(\,\Lta(\Surf)\), \(\Hoa(\Surf)\), and \(\Hooa(\Surf)\)
	are exactly the inner products
	that generate the norms
	of the spaces
	\(\Lt(\dSurf)\), \(\Ho(\dSurf)\), and \(\Hoo(\dSurf)\), respectively.
	Then \equref{distortedInnerProducts} and \lemref{boundona} imply the
	respective norm equivalences with
	\(\Lt(\Surf)\), \(\Ho(\Surf)\), and \(\Hoo(\Surf)\).
The pointwise estimates in the proof of \lemref{boundona} imply the
	equivalence of \(\Woinf(\Surf)\) and \(\Woinfa(\Surf)\),
	as well as \(\Woinf(\dSurf)\) under the map \(\Phi\).
\end{proof}

For inscribed meshes and \(\sigma=2\), \lemref{samefunctionspaces} parallels
\citet[3. Lemma]{Dziuk1988}.

\begin{remark}\label{rem:EquivDiscreteNorms}
	The equivalence of norms independently of \(h\) also
	implies a Poincar\'e inequality for the pulled back cone metric norms,
	where the constant is independent of \(h\),
	\begin{equation*}\begin{split}
		\norm{u}_\Lta &\leq C \norm{u}_\Hooa \quad\quad \forall u \in \Hoo
		\;\textrm{,} \\
		\norm{u - u_{\Surf\!\!,\,h}}_\Lta &\leq C \norm{u}_\Hooa \quad\quad
		\forall u \in \Ho
		\;\textrm{,}
	\end{split}\end{equation*}
	where \(u_{\Surf\!\!,\,h}\) is the average of \(u\) with respect to
	integration in the pulled back metric \(g_h\).
	These Poincar\'e inequalities can be derived from the smooth inequalities
	on \(\Gamma\) using \lemref{samefunctionspaces}.
\end{remark}

In order to quantify the differences between
the inner products with respect to the smooth metric \(g\) and the pulled back
metric \(g_h\),
we introduce the following bilinear forms:

\begin{definition}\label{def:differenceforms}
	The difference bilinear forms are defined as follows:
	\begin{equation}\begin{split}
		c(u,v) &\defeq \ltpr{u}{v} - \ltpra{u}{v} \;\textrm{,} \\
		d(u,v) &\defeq \hopr{u}{v} - \hopra{u}{v}
		\;\textrm{.}
	\end{split}\end{equation}
\end{definition}

\begin{lemma}\label{lem:differenceformbounds}
	The difference bilinear forms
	satisfy
	\begin{equation*}\begin{split}
		\abs{c(u,v)} &\leq C h^\sigma \norm{u}_\Lt \norm{v}_\Lt
		\quad\quad \forall u,v \in \Lt \;\textrm{,} \\
		\abs{d(u,v)} &\leq C h^\sigma \norm{u}_\Hoo \norm{v}_\Hoo
		\quad\quad \forall u,v \in \Ho
		\;\textrm{.}
	\end{split}\end{equation*}
\end{lemma}
\begin{proof}
	This is a consequence of \lemref{boundona} and the Cauchy--Schwarz
	inequality.
\end{proof}

\subsection{Discretizing the function spaces}
\label{sec:discretefunctionspaces}

Having discretized the geometry of \(\Surf\), we now turn to the approximation
of function spaces.

\begin{definition}
	\label{def:femspaces}
	Let \(\hatSh\)
	be the space of continuous functions on the mesh
	\(\dSurf\) that are linear within each triangle.
	On boundary triangles (which might have a curved edge), an isoparametric
	modification is applied \cite[]{Zlamal1973,Scott1973}, which projects a curved
	edge of a triangle onto a straight edge while keeping the vertices fixed and
	minimizing distortion.\footnote{This treatment of the boundary
	does not appear in the work of \citet{Wardetzky2006}.
	The main results of Wardetzky, however, remain true if an optimal isoparametric
	element is chosen \cite[]{Bernardi1989}, with minor modifications to the proofs,
	as long as Conditions \textbf{(C1-C4)} hold, see \remref{wardetzkybdry}.}
	On the resulting triangle, functions are required to be linear.
	\(\hatSho \defeq \hatSh \cap \Hoo(\dSurf) \) is the space of finite element
	functions that evaluate to zero at the boundary.

	Let \(\Sh\) be the lift of the function space \(\hatSh\) under the inverse
	of the closest point projection \(\Phi\).
	The space \(\Sh\) has domain \(\Surf\), and is a subset of \(\Ho\)
	(as a set of functions).
Analogously, we define the discrete space \(\Sho \defeq \Sh \cap \Hoo\).
\end{definition}

\begin{remark}
	Because the inner products \(\ltpra{\cdot}{\cdot}\), \(\hopra{\cdot}{\cdot}\) defined in \equref{distortedInnerProducts} arise from the mesh's
	metric lifted to the surface \(\Surf\), the Hilbert spaces
	\(\left(\Sho, \hopra{\cdot}{\cdot}\right)\),
	\(\left(\Sh, \hopra{\cdot}{\cdot} + \ltpra{\cdot}{\cdot}\right)\) are
	isometric to the corresponding Hilbert spaces generated by \(\hatSho\), \(\hatSh\) (using the metric $g_{\dSurf}$) on \(\dSurf\).

	Using inner products \equref{nonDistortedInnerProducts} derived from the smooth surface's metric $g$ leads to \emph{different}
	Hilbert spaces \(\left(\Sho, \hopr{\cdot}{\cdot}\right)\),
	\(\left(\Sh, \hopr{\cdot}{\cdot} + \ltpr{\cdot}{\cdot}\right)\) that are
	\emph{not} isometric to \(\left(\Sho, \hopra{\cdot}{\cdot}\right)\) and
	\(\left(\Sh, \hopra{\cdot}{\cdot} + \ltpra{\cdot}{\cdot}\right)\).
	This difference will lead us to formulate two different mixed finite
	element problems, \equref{discreteproblemsurface} and
	\equref{discreteproblemmesh}, each using different metrics.
\end{remark}

\begin{remark}
	The number of degrees of freedom of \(\Sh\) is the number of mesh vertices.
	The number of degrees of freedom of \(\Sho\) is the number of mesh vertices
	minus the number of boundary vertices.
\end{remark}

\begin{remark}
	The behavior of the discrete functions spaces on the boundary triangles
	is the same as in the work of \citet{Scholz1978}, who states:
	``Let \(\Sh = \Sh(\Surf)\) be the space of continuous functions which are
	linear in each triangle of \(\dSurf\) with the usual modification for the
	curved elements (\emph{see} \citet{Ciarlet1972,Zlamal1973}).''
	\cite[p.\ 2]{Scholz1978}.
\end{remark}

The discrete spaces come with various interpolation operators. 
\begin{definition}
	Let \(\Ih : \Ht \rightarrow \Sh\)
	denote the per-vertex interpolation operator,\footnote{The Sobolev embedding theorem implies that  \(\Ht(\Surf)\subseteq C^0(\Surf)\), which justifies pointwise interpolation.} i.e., \( (\Ih u)(p) = u(p)\) for all nodes.

	Moreover,
	\(\Rh, \Rha : \Ho \rightarrow \Sh\) and
	\(\Rho, \Rhao : \Hoo \rightarrow \Sho\) are the Ritz projection operators,
	\begin{equation*}\begin{split}
        \hopr{u - \Rh u}{\eta} = \hopra{u - \Rha u}{\eta} = 0 &\quad\quad \forall
        \eta \in \Sh, \; u \in \Ho \\
        \ltpr{u - \Rh u}{1} = \ltpra{u - \Rha u}{1} = 0
        \;\textrm{,}
    \end{split}\end{equation*}
    \begin{equation*}\begin{split}
        \hopr{u - \Rho u}{\xi} = \hopra{u - \Rhao u}{\xi} = 0 &\quad\quad \forall \xi \in \Sho, \; u \in \Hoo 
         \;\textrm{.}
    \end{split}\end{equation*}

    Ritz projections are a widely used tool in finite element analysis
    \cite[]{Rannacher1982,Du2011,Dziuk2013a}.
    They are an important ingredient to show convergence of the discrete
    solution on the surface to the exact solution on the surface.
\end{definition}

Standard arguments together with the equivalence of norms from
    \lemref{samefunctionspaces} yield:

\begin{lemma}\label{lem:ritzinterpstable}
    The Ritz projection is \(\Ho\)-stable, i.e.,
    \begin{equation*}\begin{split}
    \norm{\Rho u}_\Hoo, \norm{\Rhao u}_\Hoo &\leq C \norm{u}_\Hoo
        \quad\quad \forall u \in \Hoo \ , \\
        \norm{\Rh u}_\Ho, \norm{\Rha u}_\Ho &\leq C \norm{u}_\Ho 
        \quad\quad \forall u \in \Ho
        \;\textrm{.}
    \end{split}\end{equation*}
\end{lemma}

Moreover, the interpolation operators satisfy certain interpolation
inequalities.
While these results are classical for flat domains, they require more work in
the curved regime due to the presence of second derivatives
in standard interpolation estimates.

\begin{lemma}\label{lem:interpolationinequalities}
    For \(u \in \Hoo \cap \Ht\) one has
    \begin{equation}\begin{split}\label{eq:interpolationclaims}
    	\norm{u - \Ih u}_{\Hoo} &\leq C h \norm{u}_{\Ht} \ , \\
        \norm{u - \Rho u}_{\Hoo} &\leq C h \norm{u}_{\Ht} \ ,\\
        \norm{u - \Rhao u}_{\Hoo} &\leq C h \norm{u}_{\Ht}
        \;\textrm{.}
    \end{split}\end{equation}
    Analogous results hold for \(\Ho\)-functions that are nonzero at the boundary using
    the appropriate interpolation operators.

    Let \(T = \Psi(\hat{T}) \subset  \Surf\) denote a curved triangle that is the image of a flat triangle \(\hat{T} \subset \dSurf\) under the closest point projection.  Suppose that \(u \in \Hoo \) is continuous
    and lies in \(\Ht(T)\) (resp.
    \(\Wtinf(T)\)) for each such curved triangle \(T\). Then one has
    \begin{equation}\begin{split}\label{eq:interpolationclaimspcw}
    	\norm{u - \Ih u}_{\tHoo} &\leq C h \norm{u}_{\tHt} \ , \\
    	\norm{u - \Ih u}_{\tWoinf} &\leq C h \norm{u}_{\tWtinf}
    	\textrm{,}
    \end{split}\end{equation}
    where the tilde above the norm indicates a per-triangle norm, summed
    over all triangles \(T\) in the triangulation:
    \(\norm{u}_{\tHt}^2 = \sum_{T} \norm{u}_{\Ht(T)}^2\),
    \(\norm{u}_{\tWtinf} = \max_{T} \norm{u}_{\Wtinf(T)}\).
\end{lemma}
\begin{proof}
The estimate in the first line of  \equref{interpolationclaims} follows from the first line of \equref{interpolationclaimspcw}. The estimates in the second and third lines of \equref{interpolationclaims} follow from the first line of \equref{interpolationclaims} and the \(\Ho\)-stability from \lemref{ritzinterpstable}. It thus remains to show  \equref{interpolationclaimspcw}.
    
 For flat triangles, the estimates of  \equref{interpolationclaimspcw} hold by
 classical results (see, e.g., \cite[Theorem 6.4]{BraessFiniteElements} and
 \cite[Theorem 4.4.4]{BrennerScott}).
 In the curved setting, we can use the flat estimates by bounding the curved
 norms by the respective flat norms.
 The left-hand sides of \equref{interpolationclaimspcw}
 can be bounded by their flat counterparts using \lemref{boundona} and the
 definition of the metrics \(g, g_h\).
 Let \(\sigma \defeq u - \Ih u\).
 For almost all \(x \in \Surf\).
 \begin{equation*}\begin{split}
 	\abs{g(\nabla_\Surf\sigma(x), \nabla_\Surf\sigma(x))} &=
 	\abs{g_h(A(x)\nabla_\dSurf\sigma(x), \nabla_\dSurf\sigma(x))} \abs{\det A(x)}^{-\frac{1}{2}} \\
 	&\leq \abs{g_h(\nabla_\dSurf\sigma(x), \nabla_\dSurf\sigma(x))} +
 	\norm{A \abs{\det A}^{-\frac{1}{2}} - \id}_\Linf
 	\abs{g_h(\nabla_\dSurf\sigma(x), \nabla_\dSurf\sigma(x))} \\
 	&\leq
 	C \left( 1 + h^\sigma \right)
 	\abs{g_h(\nabla_\dSurf\sigma(x), \nabla_\dSurf\sigma(x))}
 	\quad\textrm{applying \lemref{boundona}} \\
 	&\leq C \abs{g_h(\nabla_\dSurf\sigma(x), \nabla_\dSurf\sigma(x))}
 	\quad\textrm{for small enough } h
 	\;\textrm{,}
 \end{split}\end{equation*}
 which can be used to bound the left-hand sides with a quantity that can be
 used in the classical flat estimate.

 The right-hand sides require a bit more work as they involve second derivatives
 on curved triangles.
Let \(\abs{\secondderiv u }\) be the norm of the Hessian of \(u\) with respect to the metric on \(\Surf\), and let \(\abs{\secondderiva u }\)
be the norm of the Hessian of \(u\) with respect to the
metric \(g_h\) pulled back from \(\dSurf\) to \(\Surf\).
The proof of Lemma 3.3.1 of \citet{Wardetzky2006} shows that if \(u\) is twice
 classically differentiable, then one has the pointwise estimate
 \begin{equation}\label{eq:HessianEstimate}
        \abs{\secondderiva \, u }^2 \leq
        C\left( \abs{\secondderiv \, u }^2
        + \abs{\grad u}^2\right) 
        \;\textrm{.}
 \end{equation}
 
Then the first line of \equref{interpolationclaimspcw} follows from the corresponding estimate in the flat case together with \equref{HessianEstimate} and the fact that smooth functions are dense in  \(\Ht(T)\).

In order to prove the second estimate in \equref{interpolationclaimspcw},
first recall that for domains that satisfy a cone condition one has
\(W^{1,\infty}= C^{0,1}\), and every Lipschitz function \(u \in C^{0,1}\) is
classically differentiable a.e. (see, e.g.,  \citet[Remark 4.2]{Heinonen2005}).
Applying this result to \(W^{2,\infty}\) shows that every member of
\(W^{2,\infty}\) is twice classically differentiable a.e., which allows for
applying the pointwise estimate \equref{HessianEstimate}. 
\end{proof}
\\

\section{Mixed Finite Elements}
\label{sec:mixedfem}

With the discrete geometry and discrete function spaces in place, we can
now turn towards discretizing the problem \equref{weakbiharmonic}
in its mixed form \equref{smoothmixedformulation}.

Using the two inner products \equref{nonDistortedInnerProducts}  and \equref{distortedInnerProducts} on \(\Surf\) leads to two discrete mixed
problems.
In practice, one solves the discrete problem \equref{intro:discreteproblemexplicitlymesh} on the mesh, which, using the inner products \equref{distortedInnerProducts}, is equivalent to 
\begin{equation}\begin{split}\label{eq:discreteproblemmesh}
    \hopra{\uth}{\xi} &= \ltpra{f}{\xi} \quad\quad \forall \xi \in \Sho
    \;\textrm{,} \\
    \hopra{\uoh}{\eta} &= \ltpra{\uth}{\eta} \quad\quad \forall \eta \in \Sh
    \;\textrm{,}
\end{split}\end{equation}
where \(\uoh \in \Sho\), \(\uth \in \Sh\), and \(f \in \Lt\).
Having formulated an equivalent problem to
\equref{intro:discreteproblemexplicitlymesh} on the surface \(\Surf\) itself,
we can now compare \(\uoh\) to \(u_1\) and \(\uth\) to \(u_2\), and attempt to
derive an error estimate.

We additionally make use of a similar discrete problem with respect to
the inner products \equref{nonDistortedInnerProducts}, i.e., 
\begin{equation}\begin{split}\label{eq:discreteproblemsurface}
    \hopr{\utht}{\xi} &= \ltpr{f}{\xi} \quad\quad \forall \xi \in \Sho
    \;\textrm{,} \\
    \hopr{\uoht}{\eta} &= \ltpr{\utht}{\eta} \quad\quad \forall
    \eta \in \Sh
    \;\textrm{,}
\end{split}\end{equation}
where \(\uoht \in \Sho\), \(\utht \in \Sh\), and \(f \in \Lt\).
This problem is only an auxiliary problem for our proof.
Its operators are never computed in practice.
If the surface has no boundary, the solutions of \equref{discreteproblemsurface}
and \equref{discreteproblemmesh} additionally have to fulfill the
zero mean property from \defref{biharmeqndefn}.

In the planar case of \citet{Scholz1978}, where \(\Surf \subseteq \Rtwo\) is a
planar domain,
the two discrete problems \equref{discreteproblemsurface} and
\equref{discreteproblemmesh} coincide.
\\

Existence and uniqueness for \equref{discreteproblemsurface} and
\equref{discreteproblemmesh} follow from an
argument by \citet[Section 7]{CiarletTheFiniteElementMethod}, which we repeat
here for convenience. 

\begin{definition}
We define the following three Hilbert spaces:
\begin{equation}\begin{split}\label{eq:mixedspacesdefn}
		\ciarlsmooth &:= \{ \tup{v_1}{v_2} \in \Hoo \times \Lt \quad |
        \quad \hopr{v_1}{\mu} = \ltpr{v_2}{\mu} \; \forall \mu \in \Ho \}
        \;\textrm{,} \\
        \ciarldiscronsurf &:= \{ \tup{v_1}{v_2} \in \Sho \times \Shlt \quad |
        \quad \hopr{v_1}{\mu} = \ltpr{v_2}{\mu} \; \forall \mu \in \Sh \}
        \;\textrm{,} \\
        \ciarldiscronmesh &:= \{ \tup{v_1}{v_2} \in \Sho \times \Shlt \quad |
        \quad \hopra{v_1}{\mu} = \ltpra{v_2}{\mu} \; \forall \mu \in \Sh \}
        \;\textrm{,}
    \end{split}\end{equation}
    where the space
    \(\Shlt\)
    is the space \(\Sh\), but with the \(\Lt\) norm
    instead of its usual \(\Ho\) norm.
\end{definition}

As an immediate consequence of Poincar\'{e}'s inequality we obtain that the resulting product norms on these spaces are equivalent to (simpler) norms that we heavily rely on going forward:

\begin{lemma}\label{lem:ciarletspacesnorm}
	The product norms on \(\ciarlsmooth, \ciarldiscronsurf, \ciarldiscronmesh\)
    are equivalent to the norms induced by the inner products
	\begin{equation}\label{eq:EquivNorms}
	\begin{split}
	\big( \tup{v_1}{v_2}, \tup{w_1}{w_2}\big) \; &\mapsto \; \ltpr{v_2}{w_2} \quad
    \text{on \(\ciarlsmooth, \ciarldiscronsurf\)} \;\textrm{,} \\
	\big(\tup{v_1}{v_2}, \tup{w_1}{w_2} \big)\; &\mapsto \; \ltpra{v_2}{w_2} \quad
    \text{on \(\ciarldiscronmesh\)}
        \;\textrm{.}
    \end{split}\end{equation}
\end{lemma}
\begin{proof}
	The symmetric bilinear forms defined by \equref{EquivNorms} are indeed
    positive definite since \(v_2 = 0\) implies \(v_1 = 0\) for all elements
    \(\tup{v_1}{v_2}\) in
    \(\ciarlsmooth, \ciarldiscronsurf, \ciarldiscronmesh\).
    Poincar\'{e}'s inequality implies that
	\begin{equation*}\begin{split}
        \norm{v_1}_\Hoo^2 &= \hopr{v_1}{v_1} = \ltpr{v_2}{v_1} \leq
        \norm{v_2}_\Lt \norm{v_1}_\Lt \leq C \norm{v_2}_\Lt \norm{v_1}_\Hoo
        \;\textrm{.}
    \end{split}\end{equation*}
    Therefore, \(\norm{v_1}_\Hoo \leq C \norm{v_2}_\Lt\), which proves the lemma
    for \(\ciarlsmooth\).
By \remref{EquivDiscreteNorms}, an identical derivation holds for
    \(\ciarldiscronsurf, \ciarldiscronmesh\).
\end{proof}

\begin{definition}
    \label{def:mixedfemfunctionals}
    On
    the linear spaces \(\ciarlsmooth, \ciarldiscronsurf, \ciarldiscronmesh\)
    we can define the functionals
\begin{equation}\begin{split}\label{eq:optimizationprobs}
        \ciarlfunc(\tup{v_1}{v_2}) &:=
        \frac{1}{2}\ltpr{v_2}{v_2} - F(\tup{v_1}{v_2})
        \quad\quad
        \textrm{on } \ciarlsmooth \textrm{ and } \ciarldiscronsurf
        \;\textrm{,} \\
        \ciarlfunca(\tup{v_1}{v_2}) &:=
        \frac{1}{2}\ltpra{v_2}{v_2} - F(\tup{v_1}{v_2})
        \quad\quad
        \textrm{on } \ciarldiscronmesh
        \;\textrm{,}
    \end{split}\end{equation}
    for a dual function
    \(F \in \ciarlsmooth',\)
    \((\ciarldiscronsurf)'\), or \((\ciarldiscronmesh)'\),
    respectively.
The explicit dual functions used to solve the mixed finite element
    problems are introduced in \lemref{uniqueminimizer}.
\end{definition}

As a direct consequence of \lemref{ciarletspacesnorm} and the Riesz
representation theorem,
the functionals from \equref{optimizationprobs} have unique minimizers:
\begin{lemma}[Existence and uniqueness for the mixed biharmonic problem]\label{lem:uniqueminimizer}
The functionals \(\ciarlfunc\) on \(\ciarlsmooth\),  \(\ciarlfunc\) on \(\ciarldiscronsurf\), and \(\ciarlfunca\) on \(\ciarldiscronmesh\) have unique minimizers \(\usmooth\defeq\tup{u_1}{u_2}\), \(\udiscronsurf\defeq\tup{\uoht}{\utht}\), and \(\udiscronmesh\defeq\tup{\uoh}{\uth}\), respectively.

Denoting by \(\rieszsmooth : \ciarlsmooth \rightarrow \ciarlsmooth'\),
\(\rieszdiscronsurf : \ciarldiscronsurf \rightarrow (\ciarldiscronsurf)'\), and
\(\rieszdiscronmesh : \ciarldiscronmesh \rightarrow (\ciarldiscronmesh)'\)
the respective Riesz maps, and using the inner products defined in
\lemref{ciarletspacesnorm}, these minimizers solve the 
mixed biharmonic equation and can be written as
\begin{equation}\begin{split}\label{eq:allrieszops}
	\rieszsmooth \usmooth = \Fsmooth \; , \quad  
	\rieszdiscronsurf \udiscronsurf = \Fdiscronsurf  \; , \quad  
	\rieszdiscronmesh \udiscronmesh = \Fdiscronmesh
	\textrm{,}
\end{split}\end{equation}
for \(\Fsmooth \tup{v_1}{v_2}\defeq\ltpr{v_1}{f}\) on \(\ciarlsmooth\),
\(\Fdiscronsurf \tup{v_1}{v_2}\defeq\ltpr{v_1}{f}\) on \(\ciarldiscronsurf \),
and \(\Fdiscronmesh \tup{v_1}{v_2}\defeq\ltpra{v_1}{f}\) on
\(\ciarldiscronmesh\).
\end{lemma}

\begin{remark}
    \label{rem:noboundonl2}
    \lemref{uniqueminimizer} ensures existence and uniqueness of the mixed
    formulation of the biharmonic equation.
    In particular, if the domain is a smooth surface with smooth boundary, then
    the respective minimizer on \(\ciarlsmooth\) solves the weak biharmonic
    equation \equref{weakbiharmonic}.
    \emph{Notice, however, that the respective minimizer on \(\ciarlsmooth\)
    does not necessarily solve the weak biharmonic equation
    \equref{weakbiharmonic} if the domain does not satisfy appropriate
    regularity conditions}, e.g., if the domain has reentrant corners
    \cite[Section 4.3]{Stylianou2010}.
    This does not pose a problem for the smooth solution \(u_1\), since we work
    with smooth surfaces for which standard regularity estimates hold.
    However, similar issues impact the discrete solutions, since 
    \lemref{uniqueminimizer} can only be used to show
    that the \(\Hoo\) norms of \(\uoht, \uoh\) are bounded independently of
    \(h\), and likewise, that the \(\Lt\) norms of \(\utht, \uth\) are bounded
    independently of \(h\).
    We \emph{cannot} infer boundedness of the \(\Ho\) norms of \(\utht, \uth\)
    independently of \(h\) -- these norms
    can (and in certain cases will)
    blow up as \(h\) decreases.
    We address this issue in the next section.
\end{remark}

\begin{table}[h]
\setlength{\extrarowheight}{6pt}
\centering
\begin{tabular}{L{66pt}|C{76pt}|C{76pt}|C{84pt}}
    name &
    \shortstack{\emph{smooth functions} \\ \emph{on the surface}} &
    \shortstack{\emph{discrete functions} \\ \emph{on the surface}} &
    \shortstack{\emph{lifted discrete functions} \\
    \emph{from the mesh}} \\[14pt]
    defined on &
    \(\Surf\) &
    \(\Surf\) &
    \shortstack{\(\Surf\) \\ (but behaves like \(\dSurf\) \\
    due to metric pullback)} \\[14pt]
    \shortstack[l]{solutions to the \\ biharm.\ equation} &
    \(u_1\), \(u_2\) &
    \(\uoht\), \(\utht\) &
    \(\uoh\), \(\uth\) \\[14pt]
    Set of functions &
    \(\Ho\) &
    \(\Sh\) &
    \(\Sh\) \\[14pt]
    \shortstack[l]{Set of functions \\ (zero at bdry.)} &
    \(\Hoo\) & 
    \(\Sho\) &
    \(\Sho\) \\[14pt]
    \shortstack{\(\Lt\) product \\ \(u,v \in \Lt\)} &
    \(\ltpr{u}{v} = \int_\Surf u v \;\dx \) &
    \(\ltpr{u}{v} = \int_\Surf u v \;\dx \) &
    \(\ltpra{u}{v} = \int_\Surf uv \left| \det A \right|^\frac12 \;\dx\) \\[16pt]
    \shortstack{\(\Ho\) product \\ \(u,v \in \Ho\)} &
    \shortstack{ \(\hopr{u}{v} = \) \\[2pt] \(\int_\Surf g\left(\grad u, \grad v\right) \;\dx\) } &
    \shortstack{ \(\hopr{u}{v} = \) \\[2pt] \(\int_\Surf g\left(\grad u, \grad v\right) \;\dx\) } &
    \shortstack{ \(\hopra{u}{v} = \) \\[-2pt] \({\scriptstyle \int_\Surf g(A^{-1} \grad u, \grad v)
    \left| \det A \right|^\frac12 \;\dx}\) } \\[14pt]
    mixed FEM space &
    \(\ciarlsmooth\) &
    \(\ciarldiscronsurf\) &
    \(\ciarldiscronmesh\) \\[14pt]
    Riesz operator &
    \(\rieszsmooth\) &
    \(\rieszdiscronsurf\) &
    \(\rieszdiscronmesh\)
\end{tabular}
\caption{A summary of the different settings, solutions, and function spaces
used in this article.
\label{tab:spacestable}}\end{table}

We end this section with a table visualizing all of our three settings,
the functions spaces used for each of them, and the solutions that each of
them contains (\tabref{spacestable}).
It is important to note here that \emph{all} functions and function spaces
are defined exclusively on the smooth surface \(\Surf\), making comparisons
between functions from different settings possible.

\section{Convergence of the Numerical Method}
\label{sec:convergence}

It is somewhat surprising that the derivatives of \(\utht, \uth\) appear in the linear systems that we are solving, but the \(\Lt\)-norms of these derivatives 
cannot be bounded independently of \(h\). This indeed complicates the task of bounding errors between solutions of \equref{smoothmixedformulation}, \equref{discreteproblemsurface}, and \equref{discreteproblemmesh}.
\citet{Scholz1978} elegantly solves this issue by utilizing the Ritz projection
in order to cancel contributions of derivatives of \(\utht\).
In the curved case, an argument similar to Scholz's' only serves to show that
\(\utht\) converges to \(u_2\), and \(\uoht\) converges to \(u_1\).
In particular, for the case of curved geometries, one also must account for
the approximation of the curved surface by a piecewise flat surface, i.e., to
show that \(\uth\) converges to \(\utht\).
This is precisely why the curved case is more intricate than the flat one.

\subsection{Convergence of the discrete problem on the mesh to the discrete
problem on the surface}
\label{sec:discreteconvergence}

In this section and
\secref{convonsurface} we treat the case of
surfaces with boundary;
the case of
surfaces without boundary
(treated later) is significantly simpler. \\

So far, our treatment  for the mixed biharmonic problem has considered the
smooth setting alongside the two discrete settings.
The next step, however, only works in the two discrete settings.
We define a discrete Laplace operator that maps into the \(\Lt\)-like
space
\(\Shlt\)
from \(\Sho\).

\begin{lemma}[Discrete Laplacians]\label{lem:DiscreteLap}
	There exist bounded linear and injective operators
	\(\discrlap, \discrlapa : \Sho \rightarrow \Shlt\)
  such that
	\begin{equation}\begin{split}
		\tup{v_1}{\discrlap v_1} \in \ciarldiscronsurf &\quad\quad \forall v_1
		\in \Sho \;\textrm{,} \\
		\tup{v_1}{\discrlapa v_1} \in \ciarldiscronmesh &\quad\quad \forall v_1
		\in \Sho
		\;\textrm{.}
	\end{split}\end{equation}
Moreover, every element in
  \(\ciarldiscronsurf\)
  can be written as the pair
	\(\tup{v_1}{\discrlap v_1}\),
  and every element in
  \(\ciarldiscronmesh\)
  can be written as the pair
  \(\tup{v_1}{\discrlapa v_1}\).
\end{lemma}
\begin{proof}
	We prove the lemma for
  \(\discrlap\);
  the proof for
  \(\discrlapa\)
  is similar.
  For all
  \(\tup{v_1}{v_2} \in \ciarldiscronsurf\)
  we have that
	\begin{equation*}
		\hopr{v_1}{\mu} = \ltpr{v_2}{\mu}
		\quad\quad \forall \mu \in \Sh
		\;\textrm{.}
	\end{equation*}
Written as a discrete linear equation, the right-hand side involves the
	mass matrix \(M\) for
  piecewise linear (except, potentially, on boundary triangles)
  Lagrange finite elements.
	This matrix is invertible. 
	We can thus define
  \(\discrlap  \defeq  M^{-1} \Sigma \),
	where \(\Sigma\) denotes the discrete Laplacian stiffness matrix with columns
  in \(\Sho\) and rows in \(\Sh\).
	The resulting operator is well-defined and linear.
Injectivity follows from the solvability of the Poisson equation.
Indeed,
  \(\hopr{v_1}{\eta} = 0 \; \forall \eta \in \Sho\)
  has a unique solution \(v_1 = 0 \in \Sho\).
  This implies that the discrete Laplacian stiffness matrix
    \(\Sigma\)
  is injective, and hence  \(\discrlap \) is injective.
	
	It remains to show that every element in
  \(\ciarldiscronsurf\)
  can be written as a pair
  \(\tup{v_1}{\discrlap v_1}\).
	In order to see that, let
  \(\tup{v_1}{v_2} \in \ciarldiscronsurf\).
	Then, the definition of
  \(\ciarldiscronsurf\)
  implies that
	\begin{equation*}
		0 = \hopr{0}{\mu} = \ltpr{v_2 - \discrlap v_1}{\mu}
		\quad\quad \forall \mu \in \Sh
		\;\textrm{.}
	\end{equation*}
	Therefore,
  \(\discrlap v_1 = v_2\). 
\end{proof}
\\

\begin{remark}
	The linear operators
  \(\discrlap, \discrlapa\)
  are bounded, as they are discrete operators.
	This bound, however, is \emph{not} independent of \(h\).
\end{remark}

The next result is central for relating solutions from the two discrete spaces.
\begin{lemma}[Inverse estimate]\label{lem:thdiffs}
	We have that
	\begin{equation*}
		\norm{\discrlap v_1 - \discrlapa v_1}_\Lt \leq C h^{\sigma-1}
		\norm{v_1}_\Hoo
		\quad\quad \forall v_1 \in \Sho
		\;\textrm{,}
	\end{equation*}
	where the constant \(C\) is independent of \(h\),
  and where \(\sigma\) was defined in \defref{sigmadefn}.
\end{lemma}
\begin{proof}
	For
  \(\mu \in \Shlt\) and \(c,d\) as defined in \defref{differenceforms},
  it holds that
	\begin{equation*}\begin{split}
		\ltpr{\discrlap v_1 - \discrlapa v_1}{\mu} &= \ltpr{\discrlap v_1}{\mu}
		- \ltpra{\discrlapa v_1}{\mu} - c(\discrlapa v_1, \mu)
		= \hopr{v_1}{\mu} - \hopra{v_1}{\mu} - c(\discrlapa v_1, \mu) \\
		&= d(v_1,\mu) - c(\discrlapa v_1, \mu)
		\leq Ch^\sigma \norm{v_1}_\Hoo \norm{\mu}_\Hoo
		+ Ch^\sigma \norm{\discrlapa v_1}_\Lt \norm{\mu}_\Lt
		\;\textrm{.}
	\end{split}\end{equation*}
Using the standard inverse estimate, we have that
	\(\norm{\discrlapa v_1}_\Hoo \leq Ch^{-1} \norm{\discrlapa v_1}_\Lt\)
  \cite[II 6.8]{BraessFiniteElements}.
	Therefore,
	\begin{equation*}\begin{split}
		\norm{\discrlapa v_1}_\Lt^2 & \leq
		C \norm{\discrlapa v_1}_\Lta^2 = C \hopra{v_1}{\discrlapa v_1}
		\leq C h^{-1} \norm{v_1}_\Hoo \norm{\discrlapa v_1}_\Lt
		\;\textrm{,} 
\end{split}\end{equation*}
	which proves the lemma
  after applying another inverse estimate to \(\mu\). 
	\end{proof}
\\

Using the maps
\(\discrlap, \discrlapa\)
allows for constructing a map that relates the two spaces
\(\ciarldiscronsurf\) and \(\ciarldiscronmesh\).
\begin{lemma}\label{lem:DefW}
	Let
  \(\discrspacesmap : \ciarldiscronsurf \rightarrow \ciarldiscronmesh\)
	be the linear map such that
	\begin{equation}\label{eq:def-Wh}
		\discrspacesmap \left(\tup{v_1}{\discrlap v_1}\right) =
		\tup{v_1}{\discrlapa v_1}
		\;\textrm{.}	
	\end{equation}

  \(\discrspacesmap\)
  is well-defined, bounded independently of \(h\), invertible, and the inverse
  is also bounded independently of \(h\).
  Additionally,
  \begin{equation}\label{eq:lhbdbylhwave}
    \norm{(\discrspacesmap)^* \ \discrspacesmap - \id} \leq
    C h^{\sigma-1}
    \;\textrm{,}
  \end{equation}
  where \(\left(\discrspacesmap\right)^*\) is the linear adjoint of
  \(\discrspacesmap\).
\end{lemma}
\begin{proof}
	Well-definedness follows from the fact that every element in
  \(\ciarldiscronsurf\)
  can uniquely be written as a pair
	\(\tup{v_1}{\discrlap v_1}\),
  and every element in
  \(\ciarldiscronmesh\)
  can uniquely be written as a pair
  \(\tup{v_1}{\discrlapa v_1}\), see \lemref{DiscreteLap}.
  \(\discrspacesmap\)
  is invertible, as the inverse mapping is given by
	\(\tup{v_1}{\discrlapa v_1} \mapsto \tup{v_1}{\discrlap v_1}\).

	We now show
  boundedness of \(\discrspacesmap\); 
	a similar argument works to show boundedness of  the inverse.
	Let
  \(\tup{v_1}{\discrlap v_1} \in \ciarldiscronsurf\).
  Then
  \begin{equation*}\begin{split}
    \norm{\discrspacesmap\left(\tup{v_1}{\discrlap v_1}\right)} &=
    \norm{\discrlapa v_1}_\Lta
    \leq C \norm{\discrlapa v_1}_\Lt \leq C \norm{\discrlap v_1}_\Lt +
    C \norm{\discrlapa v_1 - \discrlap v_1}_\Lt \\
    &\leq C \norm{\discrlap v_1}_\Lt + C h^{\sigma-1} \norm{v_1}_\Hoo
    \textrm{,}
  \end{split}\end{equation*}
  where the last inequality follows from \lemref{thdiffs}.
  Using that \(\sigma > 1\) and the equivalence of norms from
  \lemref{ciarletspacesnorm},
  \(\discrspacesmap\) is bounded independently of \(h\).
  \\

We have 
  \begin{equation*}
    \norm{(\discrspacesmap)^* \ \discrspacesmap - \id}
  = \sup_{\substack{\tup{v_1}{\discrlap v_1} \in \ciarldiscronsurf }}\frac{1}{\norm{\discrlap v_1}_{\Lt}^{2}}
  \abs{\ltpra{\discrlapa v_1}{\discrlapa v_1}
    - \ltpr{\discrlap v_1}{\discrlap v_1}}
    \;\textrm{,}
  \end{equation*}
  where, again, we have used the equivalence of norms from
  \lemref{ciarletspacesnorm}.
Now,
  \begin{equation}\begin{split}\label{eq:bcwstarwproof}
    \abs{\ltpra{\discrlapa v_1}{\discrlapa v_1}
    \!\! - \ltpr{\discrlap v_1}{\discrlap v_1}}
&= \abs{\ltpra{\discrlapa v_1}{\discrlapa v_1 - \discrlap v_1}
     + \ltpra{\discrlapa v_1}{\discrlap v_1}
    \!\! - \ltpr{\discrlap v_1}{\discrlap v_1}} \\
    &= \abs{\ltpra{\discrlapa v_1}{\discrlapa v_1 - \discrlap v_1}
    + \hopra{v_1}{\discrlap v_1}
    - \hopr{v_1}{\discrlap v_1}} \\
    &\leq \norm{\discrlapa v_1}_\Lta \norm{\discrlapa v_1 - \discrlap v_1}_\Lta
    + \abs{d(v_1, \discrlap v_1)}
    \;\textrm{,}
  \end{split}\end{equation}
  where \(d\) is the bilinear difference form from \defref{differenceforms}, and
  where we used the definition of \(\discrlap, \discrlapa\) from
  \lemref{DiscreteLap}.
By the equivalence of norms from \lemref{samefunctionspaces}, 
  \lemref{thdiffs}, and \equref{lhbdbylhwave},
  \begin{equation}\label{eq:firstsubpartwstarwproof}
    \norm{\discrlapa v_1}_\Lta \norm{\discrlapa v_1 - \discrlap v_1}_\Lta
    \leq C\norm{\discrlapa v_1}_\Lta \norm{\discrlapa v_1 - \discrlap v_1}_\Lt
    \leq Ch^{\sigma-1} \norm{\discrlap v_1}_\Lt \norm{v_1}_\Hoo
    \;\textrm{.}
  \end{equation}
  By \lemref{differenceformbounds} as well as the standard inverse estimate
  for discrete Lagrangian functions that states
  \(\norm{w}_\Hoo \leq Ch^{-1} \norm{w}_\Lt \;\; \forall w \in \Sh\)
  \cite[II 6.8]{BraessFiniteElements},
  \begin{equation}\label{eq:secondsubpartwstarwproof}
    \abs{d(v_1, \discrlap v_1)} \leq
    Ch^{\sigma} \norm{v_1}_\Hoo \norm{\discrlap v_1}_\Hoo
    \leq Ch^{\sigma-1} \norm{v_1}_\Hoo \norm{\discrlap v_1}_\Lt
    \;\textrm{.}
  \end{equation}
  Substituting \equref{firstsubpartwstarwproof} and
  \equref{secondsubpartwstarwproof} into \equref{bcwstarwproof} and applying
  \lemref{ciarletspacesnorm} then gives
  \begin{equation*}\begin{split}
    \abs{\ltpra{\discrlapa v_1}{\discrlapa v_1}
    - \ltpr{\discrlap v_1}{\discrlap v_1}}
    &\leq Ch^{\sigma-1} \norm{\discrlap v_1}_\Lt^2
    \;\textrm{,}
  \end{split}\end{equation*}
  which proves the lemma.
\end{proof}\\

We denote by
\((\discrspacesmap)' : (\ciarldiscronmesh)' \rightarrow (\ciarldiscronsurf)'\)
the dual operator to
\(\discrspacesmap\).
\((\discrspacesmap)'\)
is bounded independently of \(h\), it is invertible, and its inverse is bounded
independently of \(h\).
This is true by the same argument as the one for \(\discrspacesmap\).
If \(f: \ciarldiscronmesh \rightarrow \R \),
\begin{equation}\begin{split}
  \left( (\discrspacesmap)'(f) \right) \left(\tup{v_1}{\discrlap v_1}\right)
  &= (f \circ \discrspacesmap) \left(\tup{v_1}{\discrlap v_1}\right)
  = f \left( \tup{v_1}{\discrlapa v_1} \right) , \\
  \abs{(\discrspacesmap)'(f) \left(\tup{v_1}{\discrlap v_1}\right)}
  &\leq \norm{f}_{(\ciarldiscronmesh)'} \norm{\discrlapa v_1}_\Lt
  \quad\textrm{by \lemref{ciarletspacesnorm}} \\
  &\leq \norm{f}_{(\ciarldiscronmesh)'} \norm{\discrlap v_1}_\Lt
  \quad\textrm{by \lemref{thdiffs}} \\
  &\leq \norm{f}_{(\ciarldiscronmesh)'} \norm{\tup{v_1}{\discrlap v_1}}_{\ciarldiscronsurf}
  \quad\textrm{by \lemref{ciarletspacesnorm}}
  \;\textrm{.}
\end{split}\end{equation}
The operators
\(\discrspacesmap\), \((\discrspacesmap)'\), and \((\discrspacesmap)^*\)
provide the tool for relating the two discrete problems.
\\

We mention a lemma that is true for maps between Hilbert
spaces in general, independent of our particular setting:

\begin{lemma}\label{lem:rieszoperatordiffbound}
  Let \(\Lambda : X \rightarrow Y\) be a bijective bounded linear map between
  Hilbert spaces, and let \(R_X: X\rightarrow X'\), \(R_Y: Y \rightarrow Y'\) be the respective Riesz maps.  Let \(\alpha \in X'\) and  \(\beta \in Y'\) be given. Let \(\bar{x} \in X\) be such that \(R_X(\bar{x}) = \alpha\) and  let \(\bar{y} \in Y\) be such that \(R_Y(\bar{y}) = \beta\). Then
\[
    \norm{\Lambda\bar{x} - \bar{y}} \leq
    \norm{\id - \Lambda^*\Lambda} \norm{\alpha} \norm{\Lambda^{-1}}
    + \norm{\beta - (\Lambda')^{-1}\alpha}
    \;\textrm{.}
  \]
\end{lemma}

 \begin{proof}
  Let \(y \in Y\) arbitrary such that \(\norm{y}=1\), and let
  \(x \defeq \Lambda^{-1}y\).
  Then
  \begin{equation*}\begin{split}
  \hilbertpr{\Lambda\bar{x} - \bar{y}}{y} &=
    \hilbertpr{\Lambda\bar{x}}{y} - \hilbertpr{\bar{y}}{y}
= \hilbertpr{\Lambda\bar{x}}{\Lambda x} - \hilbertpr{\bar{x}}{x}
    + \alpha(x) - \beta(y) \\
&\leq \norm{\id - \Lambda^*\Lambda}\norm{\bar{x}}\norm{\Lambda^{-1}y}
    + \norm{\beta - (\Lambda')^{-1}\alpha}\norm{y}
    \textrm{,}
  \end{split}\end{equation*}
  which proves the lemma, since \(\norm{\bar{x}} = \norm{\alpha}\), and
  \(y\) was arbitrary with norm \(1\).
\end{proof}\\

We can now bound the error between
the solutions of our two discrete problems from \lemref{uniqueminimizer}.
\begin{lemma}\label{lem:discreteproblemdiff}
  Consider the following two linear problems for
  \(\Fdiscronsurf \in (\ciarldiscronsurf)'\),
  \(\Fdiscronmesh \in (\ciarldiscronmesh)'\),
  \begin{equation*}\begin{split}
      \rieszdiscronsurf \udiscronsurf = \Fdiscronsurf \quad \text{and}\quad
      \rieszdiscronmesh \udiscronmesh = \Fdiscronmesh
      \;\textrm{.}
  \end{split}\end{equation*}

  Then we have that
  \(\;\;
      \norm{\discrspacesmap \udiscronsurf - \udiscronmesh}_{\ciarldiscronmesh}
      \leq C h^{\sigma-1} \norm{\Fdiscronsurf}_{(\ciarldiscronsurf)'}
      + C\norm{\Fdiscronmesh - ((\discrspacesmap)')^{-1} \Fdiscronsurf}_{(\ciarldiscronmesh)'}
      \;\textrm{.}
  \)
\end{lemma}
\begin{proof}
  Combining \lemref{DefW} and \lemref{rieszoperatordiffbound}
  with \(X = \ciarldiscronsurf, Y = \ciarldiscronmesh\),
  the statement immediately follows.
\end{proof}
\\

The main result of this section relates the two discrete systems
\equref{discreteproblemsurface} and \equref{discreteproblemmesh} that we use
in our mixed finite element method.

\begin{theorem}\label{thm:discreteproblemdiff}
  Let \(\uoht, \utht\) solve problem \equref{discreteproblemsurface},
  and let \(\uoh, \uth\) solve problem \equref{discreteproblemmesh}.
  Then 
  \begin{equation*}
      \norm{\uth - \utht}_\Lt \leq C h^{\sigma-1} \norm{f}_\Lt
      \;\textrm{.}
  \end{equation*}
\end{theorem}
\begin{proof}
  Let
  \(\udiscronsurf = \tup{\uoht}{\utht}\),
  \(\udiscronmesh = \tup{\uoh}{\uth}\),
  \(\Fdiscronsurf(\tup{v_1}{v_2}) = \ltpr{f}{v_1}\),
  \(\Fdiscronmesh(\tup{v_1}{v_2}) = \ltpra{f}{v_1}\).
  Then
  \begin{equation*}\begin{split}
      \norm{\uth - \utht}_\Lt
      &= \norm{\discrlapa \uoh - \discrlap \uoht}_\Lt
      \leq \norm{\discrlapa \uoh - \discrlapa \uoht}_\Lt 
      + \norm{\discrlapa \uoht - \discrlap \uoht}_\Lt \\
      &\leq C
      \norm{\udiscronmesh - \discrspacesmap \udiscronsurf}_{\ciarldiscronmesh} +
      \norm{\discrlapa \uoht - \discrlap \uoht}_\Lt
      \quad\text{(by \equref{def-Wh})}  \\
&\leq C
      \norm{\udiscronmesh - \discrspacesmap \udiscronsurf}_{\ciarldiscronmesh} +
      Ch^{\sigma-1}\norm{\uoht}_\Hoo \quad\text{(by \lemref{thdiffs})}  \\
      &\leq C
      \norm{\udiscronmesh - \discrspacesmap \udiscronsurf}_{\ciarldiscronmesh} +
      Ch^{\sigma-1}\norm{\Fdiscronsurf}_{(\ciarldiscronsurf)'}
      \quad\text{(by \lemref{ciarletspacesnorm})}\\&\leq 
      C\norm{\Fdiscronmesh - ((\discrspacesmap)')^{-1} \Fdiscronsurf}_{(\ciarldiscronmesh)'} +
      C h^{\sigma-1} \norm{\Fdiscronsurf}_{(\ciarldiscronsurf)'}
      \quad\text{(by \lemref{discreteproblemdiff})}
      \;\textrm{.}
  \end{split}\end{equation*}

  It remains to deal with the right-hand sides of the last inequality.
  By the equivalence of norms,
  \(\norm{\Fdiscronsurf}_{(\ciarldiscronsurf)'} \leq C \norm{f}_\Lt\).
  For
  \(\tup{v_1}{\discrlapa v_1} \in \ciarldiscronmesh\)
  we have
  \begin{equation*}\begin{split}
    \left( \Fdiscronmesh - ((\discrspacesmap)')^{-1} \Fdiscronsurf \right)
      \left(\tup{v_1}{\discrlapa v_1}\right) &=
      \ltpra{f}{v_1} - \ltpr{f}{v_1} = c(f, v_1) \;\textrm{,}
      \quad \text{and hence,}\\
      \abs{\left( \Fdiscronmesh - ((\discrspacesmap)')^{-1} \Fdiscronsurf \right)
      \left(\tup{v_1}{\discrlapa v_1}\right)}
      &\leq C h^{\sigma} \norm{f}_\Lt \norm{v_1}_\Lt
      \quad\text{(by \lemref{ciarletspacesnorm})}  \\
      &\leq C h^{\sigma} \norm{f}_\Lt \norm{\discrlapa v_1}_\Lt
       \;\textrm{,} \\
      \norm{\Fdiscronmesh - ((\discrspacesmap)')^{-1} \Fdiscronsurf}_{(\ciarldiscronmesh)'}
       &\leq
      C h^{\sigma} \norm{f}_\Lt
      \;\textrm{.}
  \end{split}\end{equation*}
  This proves the result.
\end{proof}

\subsection{Convergence of the discrete problem on the surface to the exact
solution}
\label{sec:convonsurface}

Having successfully bounded the error between the discrete problem on the mesh
(with solution \(\tup{\uoh}{\uth}\)) and the discrete problem on the surface
(with solution \(\tup{\uoht}{\utht}\)), we move on to bounding the error between
\(\tup{\uoht}{\utht}\) and the exact solution, \(\tup{u_1}{u_2}\).
Our proof follows the roadmap laid  out by \citet{Scholz1978}.  However, we require considerable adjustments to extend this approach to 
curved surfaces. \\

We start with an extension of Scholz's Lemma to
curved surfaces,
using a theorem by \citet{Demlow2009} in place of
Scholz's use of a result by \citet{Nitsche1978}.
Demlow works in the setting of inscribed meshes of \citet{Dziuk1988} and adapts
a result by \citet{Schatz1998} from the flat setting.
Demlow's analysis can be adapted to our setting fulfilling
conditions \textbf{(C1-C4)} by using \lemref{samefunctionspaces}, \lemref{differenceformbounds}, and
\lemref{interpolationinequalities}.
Alternatively, one could also consider generalizing the \( \Linf\) estimate of
\citet{Rannacher1982} (who offer an improved Nitsche-type bound) to our setting.
For details, we refer to the work of \citet[Section 5.2.3]{Stein2020}.\begin{lemma}[Scholz's Lemma]\label{lem:scholzlemma}
	Let \(u \in \Hoo \cap \Wtinf\).
	Let \(\eta \in \Sh\).
	Then
	\begin{equation*}
		\abs{\hopr{u - \Rho u}{\eta}} \leq C \sqrt{h} \norm{u}_\Wtinf
		\norm{\eta}_\Lt
		\;\textrm{.}
	\end{equation*}
\end{lemma}
\begin{proof}
	Let \(\xi \in \Sho\) interpolate \(\eta\) on all interior vertices of the
	mesh.
	Let \(\varphi \defeq \eta - \xi\).

	By the definition of Ritz projection, we have that
	\begin{equation*}\begin{split}
		\hopr{u - \Rho u}{\eta} = \hopr{u - \Rho u}{\varphi}
		\;\textrm{.}
	\end{split}\end{equation*}
	As \(\varphi\) is only supported on the boundary triangles
	\(\mathcal{T}_\partial\),
	the last equation can be simplified to
	\begin{equation*}\begin{split}
		\abs{ \hopr{u - \Rho u}{\eta} }
		&= \abs{\sum_{T \in \mathcal{T}_\partial} \int_T
		\grad\left(u - \Rho u\right) \cdot
		\grad\varphi \;\dx}
		\leq C h^2 \norm{u - \Rho u}_\Woinf \sum_{T \in \mathcal{T}_\partial}
		\norm{\varphi}_{\Woinf(T)}
		\;\textrm{,}
	\end{split}\end{equation*}
	where we used the fact that the area of a triangle is bounded by \(C h^2\),
	where the \(C\) depends on the triangle regularity constants.

	By the standard inverse estimate we can conclude that
	\(\norm{\varphi}_{\Woinf(T)} \leq C h^{-1} \norm{\varphi}_{\Linf(T)} \).
	By definition,
	\(\norm{\varphi}_{\Linf(T)}  \leq  C \norm{\eta}_{\Linf(T)}\).
	Moreover, using a per-triangle calculation, we obtain that
	\(\norm{\eta}_{\Linf(T)} \leq C h^{-1} \norm{\eta}_{\Lt(T)} \).
	Thus we conclude
	\begin{equation*}\begin{split}
		\abs{ \hopr{u - \Rho u}{\eta} }
		&\leq C \norm{u - \Rho u}_\Woinf \sum_{T \in \mathcal{T}_\partial}
		\norm{\eta}_{\Lt(T)}
		\leq C h^{-\frac{1}{2}} \norm{u - \Rho u}_\Woinf \norm{\eta}_\Lt
		\;\textrm{,}
	\end{split}\end{equation*}
	where we used the fact that the number of triangles in
	\(\mathcal{T}_\partial\)
	is \(\sim h^{-1}\).

	The estimate by \citet[Theorem 3.2]{Demlow2009} (which carries over to our
	setting with non-inscribed triangle meshes with minor modifications) states
	that
	\begin{equation*}\begin{split}
		\norm{\Rho v}_\Woinf
		\leq C \norm{v}_\Woinf \quad \text{for all \(v \in \Woinf \)}
		\;\textrm{.}
	\end{split}\end{equation*}
	Together with \lemref{interpolationinequalities}, this leads to
    \begin{equation*}\begin{split}
        \norm{u - \Rho u }_\Woinf &\leq
        \norm{u - I_h u }_\Woinf +
        \norm{I_h u - \Rho u }_\Woinf
        = \norm{u - I_h u }_\Woinf +
        \norm{\Rho \left(I_h u - u \right)}_\Woinf
        \\
        &\leq \norm{u - I_h u}_\Woinf +
        C \norm{I_h u - u }_\Woinf
        \leq C h \norm{u}_\Wtinf
        \;\textrm{,}
    \end{split}\end{equation*}
    which proves the lemma.
\end{proof}
\\

Using this lemma we can now estimate the error in \(u_2\).
This mirrors the first part of Theorem 1 by \citet{Scholz1978}, but we
achieve a bound of order \(\sqrt{h}\) instead of Scholz's
\(\sqrt{h} \abs{\log h}^2\) due to the improved \lemref{scholzlemma}.

\begin{theorem}\label{thm:u2error}
Let \(u_1, u_2\) solve the smooth mixed biharmonic problem \equref{smoothmixedformulation}, and let \(\uoh, \uth\) solve the discrete mixed biharmonic problem \equref{discreteproblemmesh} on the mesh. Then one has
	\begin{equation*}
		\norm{u_2 - \uth}_\Lt \leq C \sqrt{h} \norm{f}_\Lt.
	\end{equation*}
\end{theorem}
\begin{proof}
	Using \lemref{interpolationinequalities}, we obtain that
	\begin{equation}\label{eq:reducedprobtoritz1}
		\norm{\utht - u_2}_\Lt
		\leq \norm{\utht - \Rh u_2}_\Lt + \norm{\Rh u_2 - u_2}_\Lt
		\leq \norm{\utht - \Rh u_2}_\Lt + \,C h \norm{u_2}_\Ht
		\;\textrm{.}
	\end{equation}
	
	Using \thmref{discreteproblemdiff} (together with the fact that
	\(\sigma \geq \frac 32\)) and \equref{reducedprobtoritz1}, we have
	\begin{equation}\label{eq:reduceprobtodiscronsurf}
		\norm{u_2 - \uth}_\Lt \leq
		\norm{\utht - u_2}_\Lt +\norm{\utht - \uth}_\Lt
		\leq \norm{\utht - \Rh u_2}_\Lt + \,C h \norm{u_2}_\Ht+ \,C \sqrt{h} \norm{f}_\Lt
		\;\textrm{.}
	\end{equation}
	Using \equref{reduceprobtodiscronsurf} and the regularity estimate for the smooth problem, we thus obtain
	\begin{equation}\label{eq:reducedprobtoritz2}
		\norm{u_2 - \uth}_\Lt
		\leq \norm{\utht - \Rh u_2}_\Lt + \,C \sqrt{h} \norm{f}_\Lt
		\;\textrm{.}
	\end{equation}

	It remains to bound $\norm{\utht - \Rh u_2}_\Lt$.
	To this end, we first note that
	\begin{equation*}\begin{split}
		\hopr{\uoht - \Rho u_1}{\utht - \Rh u_2}
		&= \hopr{\uoht - \Rho u_1}{\utht - u_2}
		= 0
		\;\textrm{,}
	\end{split}\end{equation*}
	using the definition of the Ritz projection for the first equality and using
	the smooth and discrete formulations of the mixed biharmonic problems for
	the second equality.

	Thus we can compute
	\begin{equation*}\begin{split}
		\norm{\utht - \Rh u_2}_\Lt^2 &= \ltpr{\utht - \Rh u_2}{\utht - \Rh u_2}
		- \hopr{\uoht - \Rho u_1}{\utht - \Rh u_2} \\
		&= \ltpr{u_2 - \Rh u_2}{\utht - \Rh u_2}
		+ \hopr{\Rho u_1 - u_1}{\utht - \Rh u_2}\\
		&\leq \norm{u_2 - \Rh u_2}_{\Lt}  \norm{\utht - \Rh u_2}_{\Lt}
		+ \hopr{\Rho u_1 - u_1}{\utht - \Rh u_2}
		\;\textrm{,}
	\end{split}\end{equation*}
	where we again used the (smooth and discrete) formulations of the mixed
	biharmonic problems. 
	The first of the two summands can be estimated using the estimates for
	the Ritz projection from \equref{interpolationclaims}.
	The second summand is covered by \lemref{scholzlemma} and the
	fact that \(\norm{u_1}_\Wtinf \leq C \norm{u_1}_\Hf \leq C \norm{f}_\Lt\). 
	Division by \(\norm{\utht - \Rh u_2}_\Lt\) then
	gives
	\begin{equation*}
		\norm{\utht - \Rh u_2}_\Lt \leq C \sqrt{h} \norm{f}_\Lt
		\;\textrm{.}
	\end{equation*}
	Together with \equref{reducedprobtoritz2} this
	proves the theorem.
\end{proof}
\\

It remains to compute the error in \(u_1\).
The next theorem follows the second part of Theorem 1 by
\citet{Scholz1978}, but requires additional work due to the curved geometries.
Because of \lemref{scholzlemma}, we achieve convergence of order
\(h\) here.

\begin{theorem}\label{thm:u1error}
	We have that
	\begin{equation*}
		\norm{u_1 - \uoh}_\Lt \leq C h \norm{f}_\Lt
		\;\textrm{.}
	\end{equation*}
\end{theorem}
\begin{proof}
	Since \(u_1 - \uoh \in \Hoo\), by assumption the biharmonic equation
	\(\Lap^2 w = u_1 - \uoh\)
	with zero Dirichlet and Neumann boundary
	conditions has a unique solution \(w \in \Hoo \cap \Hf\).
	As before, we use the geometers' convention that the Laplacian be
	\emph{positive} semidefinite. 

	We use the mixed biharmonic PDEs, Ritz projection,
	and integration by parts repeatedly to obtain
	\begin{equation*}\begin{split}
		\norm{u_1 - \uoh}_\Lt^2 &= \ltpr{u_1 - \uoh}{\Lap^2 w}
		=  \hopr{u_1 - \uoh}{\Lap w} \\
		&= \hopr{u_1 - \uoh}{\Lap w - \Rh \Lap w}
		+ \ltpr{u_2 - \uth}{\Rh \Lap w}
		- d(\uoh, \Rh \Lap w) + c(\uth, \Rh \Lap w) \\
&= \hopr{u_1 - \uoh}{\Lap w - \Rh \Lap w}
		+\ltpr{u_2 - \uth}{\Rh\Lap w - \Lap w}
		+ \hopr{u_2 - \uth}{w - \Rho w}\\
		&\quad - d(\uth, \Rho w) + c(f, \Rho w) -
		d(\uoh, \Rh \Lap w) + c(\uth, \Rh \Lap w)
		\;\textrm{.}
	\end{split}\end{equation*}

	Using \equref{interpolationclaims}, the first term of the last expression
	can be bounded by
	\begin{equation*}
		\abs{\hopr{u_1 - \uoh}{\Lap w - \Rh \Lap w}}
		= \abs{\hopr{u_1 - \Rho u_1}{\Lap w - \Rh \Lap w}}
		\leq C h^2 \norm{f}_\Lt \norm{w}_\Hf
		\;\textrm{.}
	\end{equation*}
	
	Using \equref{interpolationclaims} and \thmref{u2error}, the bound for the
	second term is
	\begin{align*}
		\abs{\ltpr{u_2 - \uth}{\Lap w - \Rh \Lap w}}
		&\leq \abs{\ltpr{\Rh u_2 - \uth}{\Lap w - \Rh \Lap w}
		 + \ltpr{u_2 - \Rh u_2}{\Lap w - \Rh \Lap w}} \\
		&\leq C h^{\frac 32} \norm{f}_\Lt \norm{w}_\Hf
		\;\textrm{.}
	\end{align*}

	We can bound the third term as follows,
	\begin{equation*}\begin{split}
		\abs{\hopr{u_2 - \uth}{w - \Rho w}}
		&\leq \abs{\hopr{u_2 - \Rh u_2}{w - \Rho w}} +
		\abs{\hopr{\uth - \Rh u_2}{w - \Rho w}} \\
		&\leq C h^2 \norm{u_2}_\Ht \norm{w}_\Ht +
		\abs{\hopr{\uth - \Rh u_2}{w - \Rho w}}
		\quad\textrm{ (by \equref{interpolationclaims})} \\
		&\leq C h^2 \norm{u_2}_\Ht \norm{w}_\Ht +
		C h \norm{f}_\Lt \norm{w}_\Wtinf
		\quad\textrm{ (by \lemref{scholzlemma} and \thmref{u2error})} \\
		&\leq C h \norm{f}_\Lt \norm{w}_\Hf
		\;\textrm{.}
	\end{split}\end{equation*}

	Finally, three of the remaining terms can be bounded as
	\begin{equation*}        
		\abs{d(\uoh, \Rh \Lap w)} + \abs{c(\uth, \Rh \Lap w)}
		+ \abs{c(f, \Rho w)}
		\leq C h^{\sigma} \norm{f}_\Lt \norm{w}_\Hf
		\;\textrm{,}
	\end{equation*}
	where we used \lemref{differenceformbounds} and \equref{interpolationclaims}.
	
	In order to bound the last remaining term, observe that
	\begin{equation*}\begin{split}
	\abs{d(u_2^h, \Rho w)} 
	&\leq C h^{\sigma} \norm{u_2^h}_\Hoo
        \norm{\Rho w}_\Hoo
        \leq C h^{\sigma} \norm{u_2^h}_\Hoo \norm{w}_\Hoo 	\;\textrm{,}\\
        \norm{u_2^h}_\Hoo &\leq \norm{\Rh u_2}_\Hoo
        + \norm{u_2^h - \Rh u_2}_\Hoo
        \leq C\norm{u_2}_\Hoo + C h^{-1} \norm{u_2^h - \Rh u_2}_\Lt \\
        &\leq C\norm{f}_\Lt + C h^{-\frac{1}{2}} \norm{f}_\Lt
        \leq C h^{-\frac{1}{2}} \norm{f}_\Lt \;\textrm{,}\\
        \abs{d(u_2^h, \Rho w)} &\leq C h^{\sigma - \frac 12} \norm{f}_\Lt\norm{w}_\Hoo
        	\;\textrm{,}
	\end{split}\end{equation*}
	where we used \lemref{differenceformbounds},
	\lemref{ritzinterpstable}, and \thmref{u2error}.
	
	Using that
	\(\Lap^2 w = u_1 - \uoh\),
	we obtain \(\norm{w}_\Hf \leq C \norm{u_1 - \uoh}_\Lt\).
	Together with the assumption that \(\sigma \geq \frac32\),
	these estimates show that
	\begin{equation*}\begin{split}
		\norm{u_1 - \uoh}_\Lt^2 &\leq C h \norm{f}_\Lt \norm{u_1 - \uoh}_\Lt
		\;\textrm{,}
	\end{split}\end{equation*}
	which proves the theorem.
\end{proof}\\

A simple corollary provides a convergence rate of \(h^\frac34\) for the
gradient of \(u_1\).

\begin{corollary}\label{cor:u1h1error}
	We have that
	\begin{equation*}
		\norm{u_1 - \uoh}_\Hoo \leq C h^\frac34 \norm{f}_\Lt
	\end{equation*}
\end{corollary}
\begin{proof}
	Using the mixed biharmonic problem, it follows that
	\begin{equation*}\begin{split}
		\hopr{\uoh - u_1}{\uoh - u_1} &=
		\ltpr{\uth - u_2}{\uoh - u_1} + d(\uoh, \uoh - u_1)
		- c(\uth, \uoh - u_1)
		\;\textrm{,} \\
		\norm{u_1 - \uoh}_\Hoo^2 &\leq
		C h^\frac32 \norm{f}_\Lt^2 + C h^\sigma \norm{f}_\Lt^2
		+ C h^{\sigma+1} \norm{f}_\Lt^2
		\;\textrm{,}
	\end{split}\end{equation*}
	where we applied the estimates from \lemref{differenceformbounds},
	the fact that the solution of the discrete problem is bounded, and
	going through the Ritz approximation as an intermediate.
Since we assumed that \(\sigma \geq \frac32\), this proves the corollary.
\end{proof}
\\

\subsection{The no-boundary case}
\label{sec:noboundarycase}

Here we provide the proof for the case of empty boundary, which is much simpler
than the case of a nonempty boundary.
This case is also effectively handled in the work of \citet{Elliott2019b}.

A similar analysis (with minor modifications) to the one presented in this section also yields corresponding convergence
results for the case of a surface with boundaries and boundary conditions
\(\Lap u = u = 0\) on
\(\partial\Surf\).

If there is no boundary,
the mixed formulation decouples, in the sense that \(u_{1}, u_{2} \in \Hoo\)
solve the decoupled Poisson equations 
\begin{align*}
\hopr{u_{2}}{\xi} &= \ltpr{f}{\xi} \quad\quad \forall \xi \in \Hoo \ ,\\
\hopr{u_{1}}{\eta} &= \ltpr{u_{2}}{\eta} \quad\quad \forall \eta \in \Hoo \ .
\end{align*}
Notice that unlike the case of nonempty boundaries, we here have
 \(u_{2}, \eta \in \Hoo\) instead of  \(u_{2}, \eta \in \Ho\), where \(\Hoo\) is the subspace of functions in \(\Ho\) that integrate to zero; see \equref{weakbiharmonic}.
The same pertains to the corresponding discrete formulations.
In this case, we obtain a better convergence rate:

\begin{theorem}
    It holds that
    \begin{equation*}\begin{split}
        \norm{\uth - u_2}_\Lt + h \norm{\uth - u_2}_\Hoo
        &\leq C h^2 \norm{f}_\Lt \;\textrm{,} \\
        \norm{\uoh - u_1}_\Lt + h \norm{\uoh - u_1}_\Hoo
        &\leq C h^2 \norm{f}_\Lt
        \;\textrm{.}
    \end{split}\end{equation*}
\end{theorem}
\begin{proof}
    By \citet[Theorem 3.3.3]{Wardetzky2006} it holds that
    \begin{equation*}
        \norm{\uth - u_2}_\Lt + h \norm{\uth - u_2}_\Hoo
        \leq C h^2 \norm{f}_\Lt
        \;\textrm{.}
    \end{equation*}

    To bound the error in \(\uoh\) we turn to the solution \(\nu_1 \in \Sho\)
    of the discrete Poisson problem
    \begin{equation}\label{eq:nuequation}
        \hopra{\nu_1^h}{\eta} = \ltpra{u_2}{\eta} \quad\quad \forall \eta \in \Sho
        \;\textrm{.}
    \end{equation}
    As \(\nu_1^h\) is the solution to a discrete Poisson problem, we obtain
    \begin{equation}\label{eq:pfforvh}
        \norm{\nu_1^h - u_1}_\Lt + h \norm{\nu_1^h - u_1}_\Hoo
        \leq C h^2 \norm{u_2}_\Lt \leq C h^2 \norm{f}_\Lt
        \;\textrm{.}
    \end{equation}

    As for the error between \(\uoh\) and \(\nu_1^h\), we know that
    \begin{equation}\begin{split}\label{eq:pfforuh}
        \norm{\uoh - \nu_1^h}_\Hoo^2 
        &\leq C \abs{\ltpra{\uth - u_2}{\uoh - \nu_1^h}}
        \leq C \norm{\uth - u_2}_\Lt \norm{\uoh - \nu_1^h}_\Hoo  \;\textrm{,}\\
        \norm{\uoh - \nu_1^h}_\Hoo &\leq C \norm{\uth - u_2}_\Lt
        \leq C h^2 \norm{f}_\Lt
        \;\textrm{.}
    \end{split}\end{equation}

    Combining (\ref{eq:pfforvh}) and (\ref{eq:pfforuh}), and using
    the Poincar\'e inequality,
we obtain that
    \begin{equation*}
        \norm{\uoh - u_1}_\Lt + h \norm{\uoh - u_1}_\Hoo
        \leq C h^2 \norm{f}_\Lt
        \;\textrm{,}
    \end{equation*}
    which proves the theorem.

\end{proof}

\section{Algorithm \& Experiments}
\label{sec:experiments}

The numerical method whose convergence we have shown in \secref{convergence}
is implemented by solving the following linear system for
\(\mathbf{u_1} \in \R^{\dim\hatSho}\), and
\(\mathbf{u_2} \in \R^{\dim\hatSh}\):
\begin{equation}\label{eq:maineqforalgorithm}
    \begin{pmatrix}
        0 & \mathbf{L} \\ \mathbf{L}^\intercal & -\mathbf{M}
    \end{pmatrix}
    \begin{pmatrix}
        \mathbf{u_1} \\ \mathbf{u_2}
    \end{pmatrix}
    =
    \begin{pmatrix}
        \mathbf{f} \\ 0
    \end{pmatrix}
    \;\textrm{.}
\end{equation}
Here we use the piecewise linear Lagrange basis functions \((\varphi_i)_i\)
as a basis for the finite element space \(\hatSho\), and the piecewise linear
Lagrange basis functions \((\rho_i)_i\) is the basis for the finite element
space \(\hatSh\).
The linear system \equref{maineqforalgorithm} corresponds to the discrete
problem on the discrete mesh \equref{intro:discreteproblemexplicitlymesh},
and is equivalent to the problem \equref{discreteproblemmesh}
(via the map \(\Phi\)), for which \secref{convergence} proves convergence estimates
(lifting \equref{intro:discreteproblemexplicitlymesh} from \(\dSurf\) to
\(\Surf\) results in the equivalent \equref{discreteproblemmesh}).
The solutions \(\hatuoh, \hatuth\) (which are equivalent to \(\uoh, \uth\))
are recovered via the sums
\begin{equation}\label{eq:discreterecoveringsolutions}
\hatuoh = \sum_i (\mathbf{u_1})_i \; \varphi_i \quad\text{and}\quad 
    \hatuth = \sum_i (\mathbf{u_2})_i \; \rho_i
    \;\textrm{.}
\end{equation}In \equref{maineqforalgorithm},  \(\mathbf{L}\) denotes the sparse Laplacian stiffness matrix,
\begin{equation}\label{eq:laplacianstiffnessmat}
    \mathbf{L} =\left(
        \int_{\dSurf} \dgrad \varphi_i \cdot \dgrad \rho_j \;\dx
    \right)_{i,j}
    \; \in \R^{\dim\hatSho \times \dim\hatSh}
    \;\textrm{,}
\end{equation}
\(\mathbf{M}\) denotes the mass matrix,
\begin{equation}\label{eq:massmatrixdefn}
    \mathbf{M} = \left( \int_{\dSurf} \rho_i \; \rho_j \;\dx \right)_{i,j}
    \; \in \R^{\dim\hatSh \times \dim\hatSh}
    \;\textrm{,}
\end{equation}
and \(\mathbf{f}\) is the right-hand side,
\begin{equation}\label{eq:drighthandsidedefn}
    \mathbf{f} =
    \left( \int_{\dSurf} \varphi_i \; ( f \circ \Psi ) \;\dx \right)_i
    \; \in \R^{\dim\hatSho}
    \;\textrm{,}
\end{equation}
where \(\Psi\) is defined in \defref{closestpointprojection}.

\equref{laplacianstiffnessmat}-\equref{massmatrixdefn} are the standard
stiffness and mass matrices for the Poisson equation on triangle meshes,
as they appear, e.g., in the work of \citet{Dziuk1988}.

\begin{figure}
    \centering
    \includegraphics{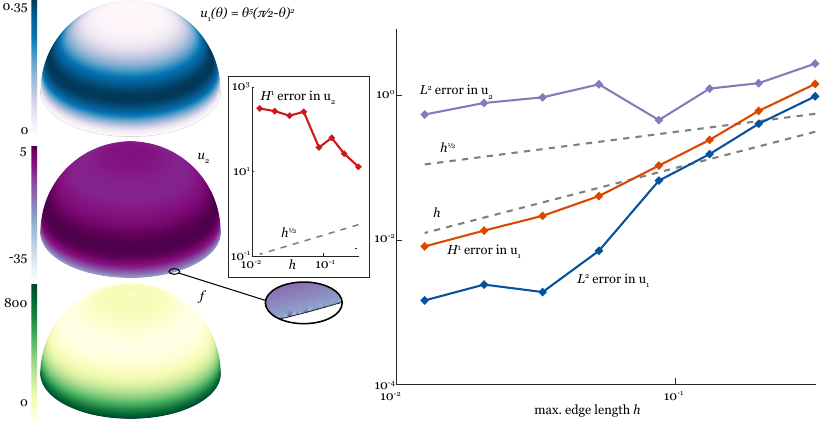}
    \caption{
    Solving the biharmonic equation on a hemisphere ($\log$-$\log$ plot).
    \label{fig:sphericalcappi2p2}}
\end{figure}

We performed a variety of numerical experiments with our method in MATLAB
R2020a,
using the gptoolbox \cite[]{gptoolbox} and triangle \cite[]{triangle}
libraries.
For all our experiments, errors are computed by measuring the difference
between the method's solution and the exact solution, projected onto the
finite element space.
The exact solutions and the right-hand sides are sampled from the exact
functions pointwise at vertices in the parameter space, and
these samples are then used as degrees of freedom of the discrete finite
element spaces.
Implementation details are provided in the supplemental material, along with
the MATLAB code used to generate the images.

\subsection{Spherical cap}
\figref{sphericalcappi2p2} shows our method
used to solve the biharmonic equation on a hemisphere.
We observe convergence to the exact solution,
\((\pi/2 - \theta)^2 \theta^5\), where \(\theta\) is the colatitude of the
spherical coordinate system, as predicted by the theory.
We observe no convergence in the \(\Ho\) norm of
\(u_2\), where the theory makes no guarantees.
Convergence to the exact solution is observed with rates at least as good as
predicted by our theory.
It can be seen in the solution plot for \(u_2\) that the function is
oscillating strongly near the boundary -- which is due to the fact that
\(\uth\) converges only in the \(\Lt\) norm, but not in the \(\Ho\) norm.
This corresponds to the theoretical intuition from \remref{noboundonl2} that
the Ciarlet theory offers no way to directly estimate the norm of the
derivative of \(\uth\):
we can only control the \(\Lt\) norm.

\begin{figure}
    \centering
    \includegraphics{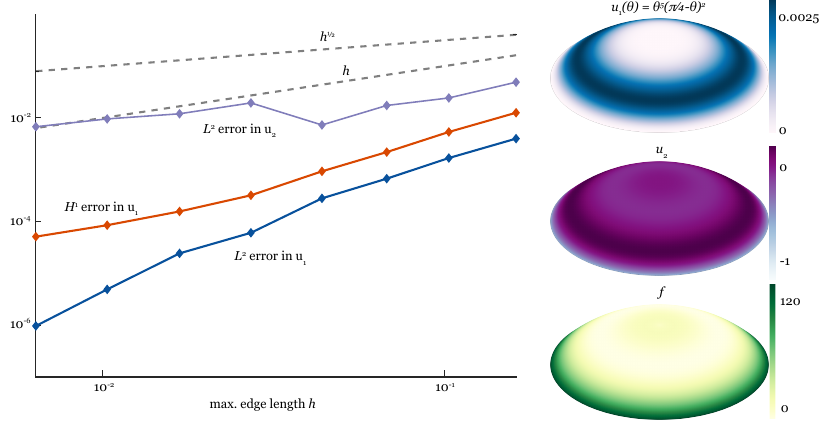}
    \caption{
    Solving the biharmonic equation on a spherical cap with colatitude from
    \(0\) to \(\pi/4\) ($\log$-$\log$ plot).
    \label{fig:sphericalcappi4}}
\end{figure}

As discussed in \remref{maybestraightboundaries}, it might be possible to relax
\condref{bijmapping} to allow for triangles with straight edges and vertices
inscribed into the boundary instead of requiring that the closest point
projection is a global bijection. 
\figref{sphericalcappi4} shows an example of this conjecture in action for a
spherical cap (not a hemisphere).\footnote{In Figures \ref{fig:sphericalcappi2p2}-\ref{fig:schwarzlantern}, the
curved boundaries of the smooth surface become straight when projected onto the
triangles, thus boundary triangles with straight edges fulfill
\condref{bijmapping}.
This is no longer true of the boundary of the spherical cap in
\figref{sphericalcappi4}.}
In this example, the
colatitude runs from \(0\) to \(\pi/4\), and
\condref{bijmapping} is not exactly fulfilled (as edges of boundary
triangles are straight).
Regardless, we observe convergence to the exact solution,
\((\pi/4 - \theta)^2 \theta^5\), where \(\theta\) is the colatitude of the spherical coordinate system.
We still empirically observe convergence,
even if \condref{bijmapping} is relaxed in the above manner.

\begin{figure}
    \centering
    \includegraphics{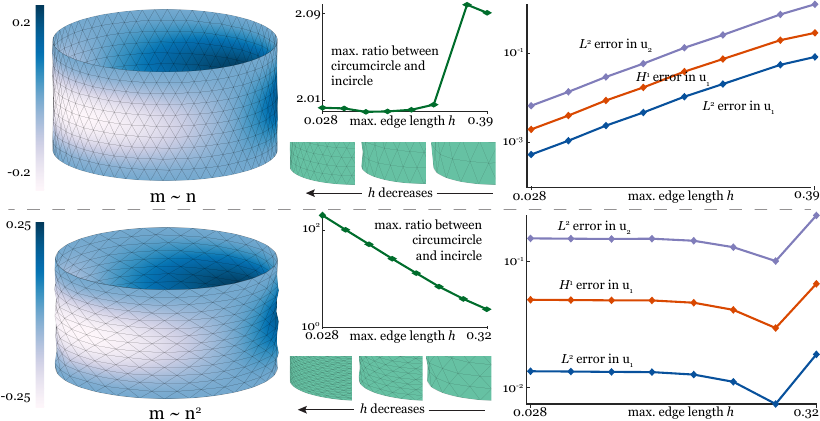}
    \caption{
    Solving the biharmonic equation on a Schwarz lantern
    perturbed on the order of \((\)mean edge length\()^2\)
    ($\log$-$\log$ plot).
    \label{fig:schwarzlantern}}
\end{figure}

\subsection{Schwarz's Lantern}
In \figref{schwarzlantern} the importance of the triangle regularity conditions
are demonstrated.
We solve the biharmonic equation with exact solution
\((\cos\varphi) (\sin \pi z) z (1-z)\),
where \(\varphi\) is the angular coordinate and \(z\) is the
\(z\)-coordinate of the cylindrical coordinate system
on a Schwarz lantern perturbed on the order of \((\)mean edge length\()^2\).
The standard Schwarz lantern fulfills conditions \textbf{(C1-C4)} if it fulfills
the triangle regularity condition.
Triangle regularity, in turn, is satisfied if \(m \sim n\),
where \(m\) is the number of vertices along the equator, and \(n\) is the
number of of vertices
along the axis of rotational symmetry.
In this case convergence is observed
(in fact, we even observe a rate of \(h^2\), which is better than predicted).
If, on refinement, \(m\) increases much
more quickly than \(n\), such as when \(m \sim n^2\), the mixed finite
element method ceases to converge.
This is a standard result, and
not at all surprising, since,
in this case, not even the surface area of the discrete mesh converges to the
surface area of the respective smooth cylinder under refinement
(see, for example, the book of \citet[Section 3.1.3]{GeneralizedCurvatures}).

\section*{Acknowledgments}

We thank Prof. Qiang Du for insightful discussion on the topic of finite
elements for curved surfaces.
We thank Abhishek Madan, Derek Liu, Henrique Maia and Anne Fleming for
proofreading.

This work is supported by the NSF under awards CCF-17-17268 and IIS-17-17178.
This work is supported by the Swiss National Science Foundation's Early
Postdoc.Mobility fellowship.
This work is partially supported by the Canada Research Chairs Program and the
Fields Centre for Quantitative Analysis and Modeling.
This work is partially supported by the DFG project 282535003:
Geometric curvature functionals: energy landscape and discrete methods.
\\

\quad\\
\quad\\
\quad\\
\quad\\
\quad\\
\quad\\
\quad\\
\quad\\
\quad\\
\quad\\
\quad\\
\quad\\
\quad\\
\quad\\
\quad\\
\quad\\
\quad\\

\pagebreak
\bibliographystyle{IMANUM-BIB}
\bibliography{mixedfem-biharmonic-surfaces}

\end{document}